\documentclass[reqno]{amsart}

\usepackage[utf8]{inputenc}
\usepackage{graphicx}
\usepackage{mathtools}
\usepackage{tikz}
\usetikzlibrary{arrows, backgrounds, calc, chains, decorations, patterns, positioning, shapes}
\usetikzlibrary{cd}
\usepackage{comment}
\usepackage{amsfonts}
\usepackage{amsmath}
\usepackage{amssymb}
\usepackage{amstext}
\usepackage{amsthm}

\usepackage{caption}
\usepackage{enumitem}
\usepackage{hyperref}
\usepackage{cleveref}
\usepackage{subcaption}
\usepackage{faktor}
\usepackage{comment}

\tikzset{
    labl/.style={anchor=south, rotate=90, inner sep=.5mm}
}

\theoremstyle{plain}

\newtheorem{theorem}{Theorem}[section]

\newtheorem{corollary}[theorem]{Corollary}
\newtheorem{lemma}[theorem]{Lemma}
\newtheorem{proposition}[theorem]{Proposition}

\theoremstyle{definition}
\newtheorem{definition}[theorem]{Definition}

\theoremstyle{remark}
\newtheorem{remark}[theorem]{Remark}
\newtheorem{example}[theorem]{Example}

\DeclareMathOperator{\Ann}{Ann}
\DeclareMathOperator{\supp}{supp}
\DeclareMathOperator{\codim}{codim}
\DeclareMathOperator{\Bl}{Bl}
\DeclareMathOperator{\rk}{rk}
\DeclareMathOperator{\Hilb}{Hilb}
\DeclareMathOperator{\Hom}{Hom}
\DeclareMathOperator{\cd}{cd}
\DeclareMathOperator{\In}{In}

\DeclareMathOperator{\p}{P}
\DeclareMathOperator{\T}{T}
\newcommand{\G}{\mathcal{G}} %building set
\newcommand{\LL}{\mathcal{L}} %geometric poset
\newcommand{\N}{\mathcal{N}} %nested set complex
\newcommand{\ZZ}{\mathbb{Z}}
\newcommand{\KK}{\mathcal{K}} %layer
\newcommand{\CC}{\mathbb{C}} %complex numbers
\newcommand{\AAA}{\mathcal{A}} %arrangement
\newcommand{\RR}{\mathbb{R}}
\newcommand{\PP}{\mathbb{P}}

\title{Cohomology Rings of Toric Wonderful Model}
\author{Lorenzo Giordani}
\author{Roberto Pagaria}
\author{Viola Siconolfi}
\date{\today}

\begin{document}

\begin{abstract}
    We describe the cohomology ring of toric wonderful model for arbitrary building set, including the case of non well-connected ones.
    Our techniques are based on blowups of posets, on Gr\"obner basis over rings and admissible functions.
\end{abstract}

\maketitle

%\tableofcontents

\section{Introduction}
%geometric poset [Bibby][Bibby, Delucchi]
We consider a $n$-dimensional algebraic torus $T$ with caracter group $X^{\ast}(T)=\Hom(T,\ZZ)$, a split direct summand $\Gamma < X^{\ast}(T) $, and any group homomorphism $\phi \colon \Gamma \rightarrow \CC^{\ast}$.
The subvariety of $T$
\[\KK_{\Gamma,\phi}:=\{t\in T \mid \chi(t) = \phi(\chi) \textnormal{, for all } \chi\in\Gamma \}\]
has codimension equal to $\rk_\ZZ(\Gamma)$.
This is a coset of the subtorus $H = \bigcap_{\chi\in\Gamma}\ker(\chi)\subseteq T$.
A \emph{toric arrangement} in $T$ is a finite collection of subtori $\AAA = \{\KK_1,\ldots,\KK_m\}$.
A subtorus of codimension one is called an hypertorus; a toric arrangement given by hypertori is said to be \emph{divisorial}.
%An arrangement is \emph{divisorial}, if all subtori are hypertori, i.e. if they have codimension $1$.
The research over recent years was focused on the topology of the complement $M(\AAA):= T \setminus \bigcup_{i=1,\dots, m} \mathcal{K}_i$ and how the topological invariants depend on the combinatorics of the toric arrangement.
A ubiquitous technique in algebraic geometry consists in finding nice compactifications. For the open variety $M(\mathcal{A})$, this was done by De Concini and Gaiffi in \cite{DeConciniGaiffi2018} by defining a \emph{projective wonderful model}.
The wonderful model depends on the choice of a toric variety and a building set, for technical reasons the researchers focused on building set $\G$ with a certain combinatorial condition, called \emph{well-connectedness}.
In this article we describe the cohomology ring of toric wonderful models for arbitrary building sets, we report an example (\ref{ex:well}) that shows how removing the well-connected hypothesis makes the choice of the building set more convenient.

Toric arrangements were introduced in the seminal article \cite{DeConciniProcesi2005}, where the additive structure of $H^*(M(\AAA); \CC)$ is presented.
Based on that article, the cohomology ring for the complement of divisorial toric arrangements were presented in \cite{CDDMP2020} and in \cite{CallegaroDelucchi17, CallegaroDelucchi24}.
The homotopy type of the complement were studied in \cite{MociSettepanella,dAntonioDelucchi12, dAntonioDelucchi15} with focus on fundamental group and minimality.
The combinatorics of divisorial toric arrangements were studied using aritmetic matroids
\cite{MociTuttePoly,DAdderioMoci,BrandenMoci,PagariaOAM,PagariaPaolini} and posets \cite{DelucchiRiedel,Bibby2022}.
The first partial compatification was introduced in \cite{Moci2012}, but the projective wonderful model was constructed in \cite{DeConciniGaiffi2018}.
%The construction depends on the choice of a \emph{building set} $\G$; for technical reasons the researchers focused on building set $\G$ with a certain combinatorial condition, called \emph{well-connectness}.
The construction depends on the choice of a 
%well connected 
\emph{building set} $\G$; for well-connected building sets the cohomology ring of the wonderful model was presented \cite{DeConciniGaiffi2019, DeConciniGaiffiPapini2020} and an additive basis was found \cite{GaiffiPapiniSiconolfi2022}.
Also the Morgan model was studied \cite{PagariaMoci, DeConciniGaiffi2021} for well-connected building sets.

Our main result (\Cref{thm:main}) consists in a presentation by generators and relations of the cohomology of the toric wonderful model $Y(X_\Delta,\G)$, that depends on the choice of a toric variety $X_\Delta$ with an \emph{equal sign property} and a building set $\G$.
Without the hypothesis that $\G$ is well-connected, the cohomology ring is not generated in degree one and our presentation is rather convoluted.

The paper is organized as follow: in \Cref{sec:blow_up_poset} we extend the definition and properties of blowups of meet-semilattices due to \cite{FeichtnerKozlov2004} to a more general class of poset given by \textit{local lattices}.
In \Cref{sec:geometric_blowups}, we prove a lemma (\Cref{thm:blowup}) on the cohomology of blowups that generalizes the one by Keel \cite{Keel92} removing a surjectivity assumption. This is necessary, because the inclusions we deal with do not induce surjective pullbacks in cohomology.
After a quick review of construction and properties of toric wonderful model (\Cref{sec:toric_wonderful_model}), the definition of our presentation $R(X_\Delta,\G)$ by generators and relations is given in \Cref{sec:presentation} and certain pullbacks and pushforward maps are computed.
In \Cref{sec:additive_basis} we provide a Gr\"obner basis for our presentation $R(X_\Delta,\G)$, however the proof is postponed to \Cref{sec:main_proof}. The set of irreducible monomials with respect to the Gr\"obner basis form an additive basis for the cohomology ring; we describe them in terms of \emph{admissible functions}. This section generalizes the results in \cite{GaiffiPapiniSiconolfi2022}.
The core of our proofs is in \Cref{sec:main_proof}, where we construct a surjective map $R(X_\Delta,\G) \twoheadrightarrow H^*(Y(X_\Delta,\G);\ZZ)$.
By using some lemmas from the previous sections we prove with deletion-contraction technique our main theorem: indeed the two $\ZZ$-modules $R(X_\Delta,\G)$ and $H^*(Y(X_\Delta,\G);\ZZ)$ have the same rank and $R(X_\Delta,\G)$ is torsion free.
In the last Section (\ref{sec:strata}), we describe the strata as a product of toric and hyperplanes wonderful model.
This observation suggests that toric wonderful models forms a module over the operad of hyperplane wonderful models.

\section{
Blowup of local lattices}
\label{sec:blow_up_poset}

In this section we recall basic notions from \cite{FeichtnerKozlov2004}, and we adapt them to our setting of toric arrangements and their underlying posets. 
The class of poset that we are going to consider are local lattices.
We denote by $x \vee y$ the set of all possible least upper bound and by $x \wedge y$ the set of all greatest lower bound of $x$ and $y$.
If the poset has a minimum then the set $x \wedge y$ is always non-empty.

\begin{definition}
Given a poset $\LL$ and $p\in\LL$ we define the
\emph{blowup of poset} at $p$:
\[ \Bl_{p} \LL = \{x \mid x \not \geq p\} \sqcup \{ (p,x,y) \mid x \not \geq p, \, y \in p \vee x\}  \]
with the ordering given by
\begin{enumerate}
    \item $x\geq x'$ if $x\geq x'$,
    \item $(p,x,y) \geq (p,x',y')$ if $x \geq x'$ and $y \geq y'$,
    \item $(p,x,y) \geq x'$ if $x \geq x'$
\end{enumerate}
 together with a projection map 
\[ \pi = \pi_\LL \colon \Bl_{p} \LL \to \LL \]
defined by $\pi(x)=x$ and $\pi((p,x,y))=y$.
\end{definition}

When the join between two elements $p,x$ is a singleton, we denote with $p\vee x$ its only element, instead of the one-element set. In the same situation, for blowups of posets, we write $(p,x)$ in place of $(p,x,p\vee x)$.

Let $\LL$ be a poset with minimum and $\p$ any poset property. We say that $\LL$ has property $\p$ \emph{locally}, if every lower interval in $\LL$ has property $\p$. For instance, we say that $\LL$ is a local lattice, local geometric lattice, or locally boolean, if every lower interval of $\LL$ is a lattice, a geometric lattice, or a boolean poset, respectively.

%\begin{lemma}
%Let $(p,x,y) \in \Bl_{p} \LL$, then 
%\[ (\Bl_{p} \LL)_{\leq (p,x,y)} \simeq \Bl_{p} (\LL_{\leq y}).\]
%\end{lemma}

%When we want to emphasize the ambient poset, we denote by $[x,y]^\LL=\{ z \in \LL \mid x \leq z \leq y\}$ the interval between $x$ and $y$ in $\LL$.

\begin{lemma}
\label{blowupislocalgeom}
If $\LL$ is a local lattice (respectively a  local geometric lattice) and $p\in \LL$, then $\Bl_{p}\LL$ is a local lattice (respectively a local geometric lattice).
\end{lemma}

\begin{proof}
    For the elements $\{x\mid x\ngeq p\}$ in the blowup, there is nothing to prove. Consider any element $(p,x,y)$ with $x\ngeq p$, and $y\in p\vee x$. Then, the following isomorphism of posets holds:
    $$[\hat0,(p,x,y)]^{\Bl_p\LL}\simeq [\hat0, x]^\LL \times B_1,$$
    where $B_1$ is the poset on 2 elements $\{\hat0< (p,\hat0)\}$. Thus, the left hand side is a (geometric) lattice, because the right hand side is a product of (geometric) lattices.
\end{proof}

%\begin{lemma}
%Let $z \in \LL$ be such that $z \geq p$, then 
%\[ \pi_{\LL}^{-1}(\LL_{\leq z}) \simeq \Bl_{p} (\LL_{\leq z}).\]
%\end{lemma}
%\begin{proof}
%    Both sides are subposets of $\Bl_p(\LL)$, thus it suffices to see that they coincide as sets.
%    We have
%    \begin{align*}
%    \Bl_p(\LL_{\leq z}) &= \{x\in \LL_{\leq z} | x \ngeq p\} \sqcup \{(p,x,y) \mid x\ngeq p, \, y\leq z\} = \pi_{\LL}^{-1}(\LL_{\leq z}). \qedhere
%    \end{align*}
%\end{proof}

A \emph{ranked poset} is a pair $(\LL,\rk)$ with $\LL$ a poset with minimum $\hat{0}$ and a rank function $\rk \colon \LL \rightarrow \mathbb{N}$ such that:
    \begin{enumerate}
        \item $\rk(\hat0) = 0$
        \item for any $X\leq Y$ in $\LL$, $\rk(X)\leq \rk(Y)$.
    \end{enumerate}
    Our definition of ranked poset does not coincide with the one used in other context and it is completely different from the standard definition of graded posets. A ranked (local) lattice is a ranked poset that is also a (local) lattice.

    The main instance of local lattices will be certain posets associated to toric (and more generally abelian) arrangements. Given a toric arrangement $\AAA$, the set of connected components of intersections of its subtori ordered by reverse inclusion is the \emph{poset of layers}, often denoted by $\LL(\AAA)$. This poset referred to as the ``combinatorics" of the arrangement, and it is its most important invariant. 
    We always endow the poset of layers with the rank function $\cd \colon \LL\rightarrow \mathbb{N}$ assigning to each layer its complex codimension in the ambient torus.
    If the toric arrangement is divisorial, then $\LL$ is a local geometric lattice, and also a \emph{geometric poset} in the sense of \cite{Bibby2022}. At the end of the article, we will exhibit an example of divisorial arrangement of subvarietes that is not an abelian arrangement, whose intersection lattice is a local geometric lattice, but not a geometric poset (\Cref{nonposetoflayers}).

\begin{definition}
\label{buildingsetdef}
A subset $\G\subset \LL\setminus\{\hat0\}$ of a lattice $\LL$ is a \emph{combinatorial building set} if for any $x\in \LL\setminus \{\hat0\}$ and $\max \G_{\leq x}=\{x_1,\ldots,x_k\}$ there is an isomorphism of posets:
\[
\phi_{x} \colon \prod_{j=1}^{k}[\hat0,x_j]\rightarrow [\hat0,x],
\]
defined by $$(y_1,\ldots, y_k)\mapsto y_1\vee\cdots\vee y_k.$$
In a ranked lattice $(\mathcal{L},\rk)$, a combinatorial building set is also \emph{geometric}, if furthermore:
$$\sum_{j=1}^{k}\rk(x_j) = \rk(x).$$
The elements of $\max \G_{\leq x}$ are called the \emph{$\G$-factors} of $x$.
\end{definition}

\begin{definition}
    Let $\LL$ be a semilattice and $\G\subset \LL$ be a building set.
    A subset $S$ of $\G$ is $\G$-\emph{nested} if, for any set of incomparable elements $x_1,\ldots,x_t$ in $S$ with $t\geq 2$, the join exists and does not belong to $\G$. 
\end{definition}

We extend the definitions of building and nested sets to local lattices.

\begin{definition}
    A \emph{combinatorial (geometric) building set} in a (ranked) local lattice $\LL$ is $\G \subseteq \LL$ such that for any $x \in \max (\LL)$ the set $\G_{\leq x}$ is a combinatorial (geometric) building set in the (ranked) lattice $\LL_{\leq x}$.
\end{definition}

\begin{definition}
A combinatorial building set $\G$ is \emph{well-connected} if for any subset $\{G_1,\ldots,G_k\}$ of elements in $\G$, if $|G_1 \vee \ldots \vee G_k|\geq 2$ then $G_1\vee\ldots \vee G_k\subset \G$. 
\end{definition}

\begin{remark}
    In \cite{BackmanDanner2024} the authors observe that in the research after the article \cite{FeichtnerKozlov2004}, maximal building set have been the preferred choice, notably also in \cite{AdiprasitoHuhKac2018}.
    In the article, they study the family of all possible building sets of meet-semilattices and their relation with convex geometries.
    In particular, they show that the family of building sets of a meet-semilattice forms a lattice and a supersolvable convex geometry \cite[Theorem 2.3, Theorem 3.2]{BackmanDanner2024}.
\end{remark}

\begin{definition}
Let $\LL$ be a ranked local lattice and $\G\subset \LL$ be a building set. A $\G$-\emph{nested set} is a pair $(S,x)$, where  $S\subseteq \G$ and $x \in \vee S$ are such that $S$ is a $\G_{\leq x}$-nested set in $\LL_{\leq x}$. 
\end{definition}

\begin{definition}
If $\LL$ is a ranked local lattice with $\G\subset \LL$ a building set and $x\in \LL$, then the $\G$-\emph{factors} of $x$ are the elements in the set
    \[ F(\G, x) = \max (\G_{\leq x}).\]
\end{definition}

\begin{comment}
\begin{remark}
\label{inheritedproperties}
    The definition of building sets and nested sets for posets specialize to the same definitions for semilattices given in \cite{FeichtnerKozlov2004}. Several properties for semilattices are immediately inherited by our generalization, for example:\\
        Let $Y\subseteq S$ subset of non-comparable elements in $S$ with $(S,x)$ $\G$-nested in $\LL$. Then, letting $y=\bigvee Y$ in $\LL_{\leq x}$, one has that 
        \begin{itemize}
            \item $(\bigvee Y,y)$ is $\G$-nested in $\LL$;
            \item The set of $\G$-factors of $y$ is $F(\G,y) = Y$
        \end{itemize}
\end{remark}
\end{comment}

We present an example on how the request of being well-connected may influence the cardinality of  a chosen building set:
\begin{example}\label{ex:well}
Consider the following divisorial toric arrangement $A(n,c)$ in $(\mathbb{C}^*)^n$ depending on a parameter $c \in \mathbb{N}_{>1}$.
Let $A(n,c)$ the collection of hypertori given by $H_1=\{t_1=1\}$ and $H_i=\{t_1t_i^c=1\}$ for $i=2,3,\dots, n$.
The minimal building set in $\G_{\min} =\{H_i\}_{i=1,\dots, n}$ consisting only in the atoms of $\LL_{A(n,c)}$ and has cardinality $n$.
However, the only well-connected building set is the maximal one; in particular, it has cardinality $\frac{(c+1)^{n}-1}{c}$.
Notice that the growth of the minimal well-connected building set is exponential in $n$ and for fixed value $n$ is unbounded, although the toric arrangement is already simple normal crossing.
\end{example}

We define a total order $\succ_{G}$ on $\G$ such that $\succ_{G}$ refines the partial order on $\G$ induced by $\LL$.

\begin{definition}
    The $\G$-\emph{nested set complex} $\N(\LL,\G)$ is the $\Delta$-complex whose $k$-faces are the $\G$-nested sets of cardinality $k+1$.
    The face maps are defined by  $d_i ((S,x)) = (T,y)$ where $S=\{S_0, S_1, \dots, S_k\}$ are ordered accordingly to $\succ_{\G}$, $T=S \setminus S_i$ and $y$ is the unique element in $\LL_{\leq x}$ such that $y \in \vee T$. 
\end{definition}

%\begin{remark}
    Recall that a $\Delta$-complex is said to be \emph{regular}, if characteristic maps on its geometric realization are homeomorphism between closed euclidean balls (or $n$-simplices) and the closed cells.
%\end{remark}

\begin{definition}
    Let $\N$ be a regular $\Delta$-complex. Its \emph{face poset} $\mathcal{F}(\N)$ is the poset consisting of the set of closed cells of $\N$, ordered by inclusion.
\end{definition}

\begin{definition}
Define $ \Bl_\G \LL$ as the iterated blowup  of elements in $\G$ in any order refining the partial order on $\G$ with the natural projection $\pi=\pi_\G \colon \Bl_\G \LL \to \LL$.
\end{definition}
The poset $\Bl_\G \LL$ does not depend on the choice of the linear order on the building set $\G$, we will prove this fact in \Cref{BlGisFP}.

\begin{proposition}
\label{blowupmaxbuilding}
Let $\LL$ be a locally geometric poset, $\G$ a building set of $\LL$ and $p\in \max(\G)$. Then,
$\tilde{\G}=(\G\setminus\{p\})\cup \{(p,\hat{0})\}$ is a building set of $\Bl_{p}\LL$. Furthermore, the $\tilde{\G}$-nested subsets are exactly the $\G$-nested subsets with $p$ replaced by $(p,\hat{0})$.
\end{proposition}

\begin{proof}
    The proof is analogous to the one in \cite{FeichtnerKozlov2004} adapted to our setting of local lattices.
    First, let us see that $\tilde{\G}$ is a building set for the blowup $\Bl_p\LL$, which is a local lattice by \Cref{blowupislocalgeom}. Let $\xi\in\max\Bl_p\LL$. If $\xi\in\{x\in\LL \mid x\ngeq p\}$ then its $\tilde{\G}_{\leq \xi}$-decomposition in $(\Bl_p\LL)_{\leq \xi}$ is clearly the same as its $\G_{\leq \pi(\xi)}$-decomposition in $\LL_{\leq \pi(\xi)}$. Suppose that $\xi=(p,x,y)$, where $x\ngeq p$ and $y\in x \vee p$ in the poset $\LL$, the following isomorphism of posets holds: $$\left[\hat0, (p,x,y)\right]^{\Bl_p\LL} \simeq [\hat0, x]^{\LL} \times B_1,$$
    where $B_1$ is the poset given by the two elements $\{\hat{0} < (p,\hat0)\}$, hence we are done by the decomposition of $[\hat{0}, x]$ in $\LL$.\\
    Let $(S,x)$ with $x\in\bigvee S$ be a $\G$-nested set in $\LL$, that is, $S$ is $\G_{\leq x}$-nested set in the lattice $\LL_{\leq x}$. We want to see that the $\tilde{\G}$-nested sets of $\Bl_p\LL$ are all and only those of the form $(\tilde{S},\tilde{x})$ where \begin{align*}
        \tilde{S} & :=\begin{cases}
        (S\setminus\{p\})\cup\{(p,\hat0)\}\text{, if } p\in S\\
        S\text{, otherwise.}
        \end{cases}\\
        \tilde{x} &= \begin{cases} ((p,\hat0)\vee\bigvee (S\setminus\{p\}))_{\leq x} \text{, if } p\in S \\ x \text{, otherwise.} \end{cases}
        \end{align*}
    \begin{description}
        \item[Case 1] 
    Suppose $(S,x)$ nested with $p\notin S$. Let $x_1,\ldots, x_t$ be incomparable elements in $S$, for $t\geq 2$. By $\G_{\leq x}$-nestedness, $(\bigvee_{i=1}^t x_i)_{\leq x}$ exists in $\LL_{\leq x}$ and does not belong to $\G_{\leq x}$. Because $p$ is not a factor and it is maximal in $\G$, we have $(\bigvee_{i=1}^t x_i)_{\leq x} \ngeq p$. Thus, the same element exists in $\Bl_p(\LL)_{\leq x}$ and does not belong to $\tilde{\G}$. We conclude that $(S,x)$ is $\tilde{\G}$-nested in $\Bl_p\LL$. Conversely, any $\tilde{\G}$-nested set $(\tilde{S},\tilde{x})$ such that $(\hat0, p)\notin \tilde{S}$ is immediately a $\G$-nested.\\
    \item[Case 2]
    Suppose $(S,x)$ nested with $p\in S$. We want to see that $(\tilde{S}, \tilde{x})$ as above is $\tilde{\G}$-nested. Subsets of incomparable elements not containing $(p,\hat0)$ can be dealt with as in Case 1.
    Let $(p,\hat0),x_1,\ldots, x_t$ be incomparable in $\tilde{S}$. As before, we get that $(\bigvee_{i=1}^t x_i)_{\leq x} \ngeq p$.
    By nestedness of $(S,x)$, the element $(p\vee \bigvee_{i=1}^t x_i)_{\leq x}$ exists in $\LL_{\leq x}$ and does not belong to $\G_{\leq x}$. Therefore, the element
    $$(p,(\bigvee_{i=1}^t x_i)_{\leq x})_{\leq x} = ((p,\hat0)\vee (\bigvee_{i=1}^t x_i))_{\leq x}$$
    exists in $(\Bl_p\LL)_{\leq\tilde{x}}$ and does not belong to $\tilde{\G}_{\leq\tilde{x}}$. We deduce that $\tilde S$ is $\tilde{\G}_{\leq\tilde{x}}$-nested, and hence  $(\tilde{S},\tilde{x})$ is $\tilde{\G}$-nested.
    An analogous argument shows that if $(\tilde{S},\tilde{x})$ is $\tilde{\G}$-nested, then $(S,x)$ is $\G$-nested \qedhere
    \end{description}
\end{proof}

\begin{proposition}
\label{BlGisFP}
Let $\G = \{G_1,\dots, G_t\}$ be a building set for $\LL$ ordered such that $G_i > G_j$, if $i<j$, that is, ordered by a linear refinement of the opposite partial order on $\LL$. Then: 
\[\Bl_\G(\LL)=\Bl_{G_t} (\Bl_{G_{t-1}} (\dots \Bl_{G_1}(\LL) \dots )) \simeq \mathcal{F}(\N(\LL,\G)).\]
\end{proposition}

\begin{proof}
Again, the proof is adapted from \cite{FeichtnerKozlov2004}. By \cref{blowupmaxbuilding}, the set $$\tilde{\G}=\{(G,\hat0)\colon G\in\G\}$$ of the transforms of the elements of $\G$ is a building set for $\Bl_\G\LL$.
By construction, it is also the set of atoms. Consider an element $x\in\Bl_\G(\LL)\setminus\{\hat0 \}$.
By \Cref{blowupmaxbuilding}, $\tilde{\G}_{\leq x}$ is $\tilde{\G}_{\leq x}$ -nested in $(\Bl_{\G}\LL)_{\leq x}$, and therefore $(\tilde{\G}_{\leq x},x)$ is $\tilde{\G}$-nested. 
Since the interval $[\hat0, x]=(\Bl_\G\LL)_{\leq x}$ decomposes as product of rank $1$ posets of the form $\{\hat0<(G,\hat0)\}$ for $G$ belonging to some nested set, we conclude that $\Bl_\G\LL$ is the face poset of the nested set complex.
\end{proof}

\begin{comment}
Again, the proof is adapted from \cite{FeichtnerKozlov2004}. With the chosen linear order on the building set we can apply \Cref{blowupmaxbuilding}. We obtain that the transforms of the elements of $\G$ (which we identify with $\G$) form a building set in $\Bl_\G(\LL)$ and are by construction the atoms of the blown up poset. Let $x\in\Bl_\G(\LL)\setminus\{\hat0 \}$: by \Cref{blowupmaxbuilding}, $\G_{\leq x}$ is $\G_{\leq x}$-nested in $(\Bl_{\G}\LL)_{\leq x}$, and therefore $(\G_{\leq x},x)$ is $\G$-nested. Because the elements of $\G$ are the atoms of $\Bl_{\G}(\LL)$, the lower intervals $(\Bl_{\G}\LL)_{\leq x}$ are product of rank $1$ boolean lattices of the form $\{\hat0 < G\}$ for some set of $\G$-nested elements of $\G$. We conclude that $\Bl_\G\LL$ is the face poset of the nested set complex.
\end{comment}

\begin{remark}\label{remark:BlG}
Every element of $ \Bl_\G \LL$ corresponds to a nested set of $\G$.
The atoms are exactly the $\G$-nested sets $(\{G\},G)$ of cardinality one.
\end{remark}

\begin{corollary}
    If $\G$ is a building set, then $\Bl_\G \LL$ is locally boolean.
\end{corollary}
\begin{proof}
    The face poset of a a regular $\Delta$-complex is locally boolean.
\end{proof}

%\begin{corollary}
%    If $g \in \Bl_{\G}(\LL)$ is an atom and $a \in \Bl_{\G}(\LL)$ such that $\pi(g) > \pi(a)$.
%    Then the two elements have a unique join, i.e.\ $\lvert a \vee g \rvert =1$.
%\end{corollary}
%\begin{proof} \todo{sistemare la dimostrazione! o rimuovere il corollario.}
%    We have $\pi(a \vee g)= \pi(a) \vee \pi(g) = \pi(g)$. 
%    In particular $a\vee g$ is a singleton.
%\end{proof}

%\begin{remark}
    If the building set $\G$ is well-connected then $\Bl_\G \LL$ is a semilattice.
%\end{remark}

\begin{lemma} \label{lemma:building_in_contr}
    Let $X \in\LL\setminus\{\hat0\}$ and $\G$ be a building set. Then, the set 
    \[ \G_X := \{Z \in \LL \mid \exists \,  G \in \G_{\not \leq X} \textnormal{ such that } Z \in G \vee X\} \]
    is a building set for $\LL_{\geq X}$.
    Moreover, let $x\in\Bl_{\G}(\LL)$ be a maximal element such that $\pi(x)=X$, that is, $x=(S,X)$ is $\G$-nested with $S$ maximal, then $$\Bl_\G(\LL)_{\geq (S,X)}\simeq \Bl_{\G_X}(\LL_{\geq X}).$$
\end{lemma}
\begin{proof}

    We first show that $\G_X$ is building set for $\LL_{\geq X}$. Consider $\xi\in\max\LL_{\geq X}$, and let $F:=F(\G,\xi)$ be the set of $\G$-factors of $\xi$ in $\LL$, then we have the decomposition in the lattice $\LL_{\leq \xi}$: $$[\hat0, \xi] \simeq \prod_{\eta\in F}[\hat0, \eta].$$
In the product, denote with $\overline{X}=(X_{\eta})_{\eta\in F}$ the element corresponding to $X$. Let also $F':=F\setminus F(\G,X)$, and $F'' := \max(\G_X)_{\leq \xi}$. We want to see that
$$[X, \xi] \simeq \prod_{\zeta\in F''}[X, \zeta].$$ We have the following isomorphism of posets:
    \begin{align*}
        [X,\xi] &= [\hat0,\xi]_{\geq X}\\
        &\simeq \left(\prod_{\eta\in F}[\hat0,\eta]\right)_{\geq \overline{X}}\\
        &\simeq \prod_{\eta\in F'}[X_{\eta}, \eta]\\
        &\simeq \prod_{\eta\in F'}[X, \eta\vee X]\\
        &\simeq \prod_{\zeta\in F''}[X,\zeta].
    \end{align*}
 In the last isomorphism we identify the joins $\eta\vee X$ with the $\G_X$-factors $\zeta$ of $\xi$.
 Indeed, by definition $\eta\vee X\in (\G_X)_{\leq \xi}$ for $\eta\in F'$, and their join equals $\xi$. On the other hand, if one of these elements was not maximal, then the corresponding $\G$-factor $\eta$ of $\xi$ in $\LL_{\leq\xi}$ would not be maximal, which is absurd.
 
    Now, we show that the two blowup posets are isomorphic: by \Cref{BlGisFP} it is enough to construct a bijection between $\G$-nested sets bigger than $(S,X)$ in $\Bl_\G(\LL)$, and $\G_X$-nested sets in $\LL_{\geq X}$.
    Let $(T,Y)>(S,X)$ be $\G$-nested. We claim that the map $f \colon \Bl_{\G}(\LL)_{\geq (S,X)} \rightarrow \Bl_{\G_X}(\LL_{\geq X})$ sending $(T,Y)$ to the $\G_X$-nested $(T^X, Y)$ with $T^X = \{(t\vee X)_{\leq Y}\mid t \in T\setminus \G_{\leq X} \}$ is the sought bijection. We want to contruct an inverse to this map.
    Consider the $\G_{X}$-nested set $(R,Y)$ and notice that for $r\in R$, the set $\{s\in \G \mid r\in(s\vee X)_{\leq Y}\}$ has a unique maximal element, which we call $r_X$. Indeed, $r_X$ is the only element in $F(\G_{\leq Y},r)\setminus F(\G_{\leq Y},X)$.
    The inverse map $g \colon \Bl_{\G_X}(\LL_{\geq X})\rightarrow \Bl_{\G}(\LL)_{\geq (S,X)}$ sends $(R,Y)$ to $(R_X,Y)$ with $R_X = S \cup \{r_X\mid r\in R\}$.
    It is clear that $fg$ is the identity on $\Bl_{\G_X}(\LL_{\geq X})$, as $(r_X\vee X)_{\leq Y}=r$. Conversely, by maximality of $S$, we have $T\setminus \G_{\leq X} = T\setminus S$, and we conclude that $gf$ is the identity by $((t\vee X)_{\leq Y})_{X} = t$.
\begin{comment} sends the $\G_X$-nested $(T,B)$ to the $\G$-nested $(T\wedge X, B)$, where $T\wedge X = \{t'\mid t'\sqcup F(\G,X) =  F(\G,t)\} \sqcup F(\G,X)$.
\end{comment}
\begin{comment} then, $S$ contains $F(\G,a)$, 
    and the following is the bijection we were looking for:
$$(S,x) \mapsto \left(\{\eta\vee a \mid \eta \in S\setminus F(\G,a)\},x\right).$$
\end{comment}
\end{proof}

\begin{remark}
    For the maximal building set $\G=\LL\setminus\{\hat0\}$, the poset $\Bl_\G(\LL)_{\geq (S,X)}$ is the face poset of the link of the simplex corresponding to $X$ in the nested set complex $\mathcal{N}(\LL,\G)$. 
\end{remark}

\begin{example}
\label{running}
    Consider the toric arrangement $\AAA$ in $T\simeq (\CC^*)^3$ with coordinates $x,y,z$ given by the three tori:
    \begin{align*}
    a &:  x = 1, \\
    b &:  x = y^3, \\
    c &:  \begin{cases} x = z \\ x^2=y^3 \end{cases}.
    \end{align*}
    This arrangement has 10 layers and a minimal building set with 6 layers.
    The poset of layers $\LL = \LL(\AAA)$ is represented in \Cref{fig:poset_layer_es}, the element of the minimal building set $\G = \G_{\min}$ are draw in blue.
\begin{figure}
    \centering
\begin{tikzpicture}[scale=0.7]
\node (min) at (0,-1) {$\hat0$};
\node [text = blue] (A) at (-4,0) {$a$};
\node [text = blue] (B) at (0,0) {$b$};
\node [text = blue] (C) at (6,2) {$c$};

\node (L1) at (-5,2) {$L_1$};
\node (L2) at (-2,2) {$L_2$};
\node (L3) at (1,2) {$L_3$};

\node[text = blue] (P1) at (-2,4) {$P_1$};
\node[text = blue] (P2) at (0,4) {$P_2$};
\node[text = blue] (P3) at (2,4) {$P_3$};

\draw (min) -- (A);
\draw (min) -- (B);
\draw (min) -- (C);

\draw (A) -- (L1) -- (B) -- (L2) -- (A) -- (L3) -- (B) -- (L2);

\draw (L1) -- (P1) -- (C) -- (P2) -- (L2);
\draw (C) -- (P3) -- (L3);

\end{tikzpicture}
    \caption{Poset of layers of the toric arrangement $\AAA$.}
    \label{fig:poset_layer_es}
\end{figure}
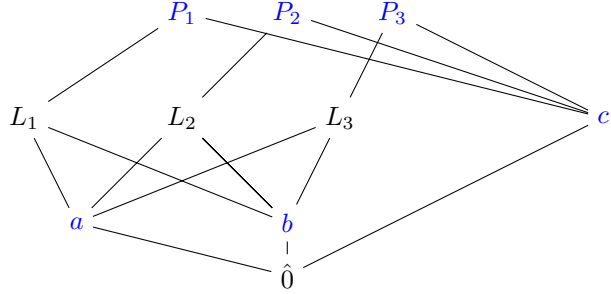

The minimal well-connected building set coincides with the maximal building set.
Moreover, the nested set complex $n(\LL,\G)$ is not a simplicial complex: indeed, nested sets $(\{a,b\},L_i)$ for $i=1,2,3$ are parallel edges between the vertices $a$ and $b$. We also depict the face poset of the nested set complex (see \Cref{fig:NestedFacePoset}), which coincides with the blown up poset $\Bl_\G \LL$ after a proper relabeling.

\begin{figure}
    \centering
    \begin{tikzpicture}[scale = 1]
    \node (0) at (-1,0) {$\scriptstyle \emptyset$};
    \node (A) at (0,1) {$\scriptstyle a$};
    \node (B) at (2,1) {$\scriptstyle b$};
    \node (P1) at (-6,1) {$\scriptstyle P_1$};
    \node (P2) at (-4,1) {$\scriptstyle P_2$};
    \node (P3) at (-2,1) {$\scriptstyle P_3$};
    \node (C) at (4,1) {$\scriptstyle c$};

    \node (P1A) at (-7,3) {$\scriptstyle \{a,P_1\}$};
    \node (P1B) at (-6,3) {$\scriptstyle \{b,P_1\}$};
    \node (P1C) at (-5,3) {$\scriptstyle \{c,P_1\}$};

    \node (P2A) at (-4,3) {$\scriptstyle \{a,P_2\}$};
    \node (P2B) at (-3,3) {$\scriptstyle \{b,P_2\}$};
    \node (P2C) at (-2,3) {$\scriptstyle \{c,P_2\}$};

    \node (P3A) at (-1,3) {$\scriptstyle \{a,P_3\}$};
    \node (P3B) at (0,3) {$\scriptstyle \{b,P_3\}$};
    \node (P3C) at (1,3) {$\scriptstyle \{c,P_3\}$};

    \node (ABL1) at (2.2,3) {$\scriptstyle (\{a,b\},L_1)$};
    \node (ABL2) at (3.6,3) {$\scriptstyle (\{a,b\},L_2)$};
    \node (ABL3) at (5,3) {$\scriptstyle (\{a,b\},L_3)$};

    \node (ABP1) at (-3,4) {$\scriptstyle \{a,b,P_1\}$};
    \node (ABP2) at (-1,4) {$\scriptstyle \{a,b,P_2\}$};
    \node (ABP3) at (1,4) {$\scriptstyle \{a,b,P_3\}$};

    \draw (P1A) -- (P1) -- (P1B);
    \draw (P2A) -- (P2) -- (P2B);
    \draw (P3A) -- (P3) -- (P3B);
    \draw (0) -- (P1) -- (P1C) -- (C) -- (P2C) -- (P2) -- (0);
    \draw (0) -- (C) -- (P3C) -- (P3) -- (0);
    \draw (0) -- (A) -- (ABL1) -- (B) -- (ABL3) -- (A) -- (ABL2) -- (B);
    \draw (P1A) -- (A) -- (P2A);
    \draw (A) -- (P3A);
    \draw (P1B) -- (B) -- (P2B);
    \draw (0) -- (B) -- (P3B);

    \draw (P1A) -- (ABP1) -- (P1B);
    \draw (P2A) -- (ABP2) -- (P2B);
    \draw (P3A) -- (ABP3) -- (P3B);
    \draw (ABL1) -- (ABP1);
    \draw (ABL2) -- (ABP2);
    \draw (ABL3) -- (ABP3);

    \end{tikzpicture}
    \caption{Face poset of the complex $n(\LL,\G)$.}
    \label{fig:NestedFacePoset}
\end{figure}
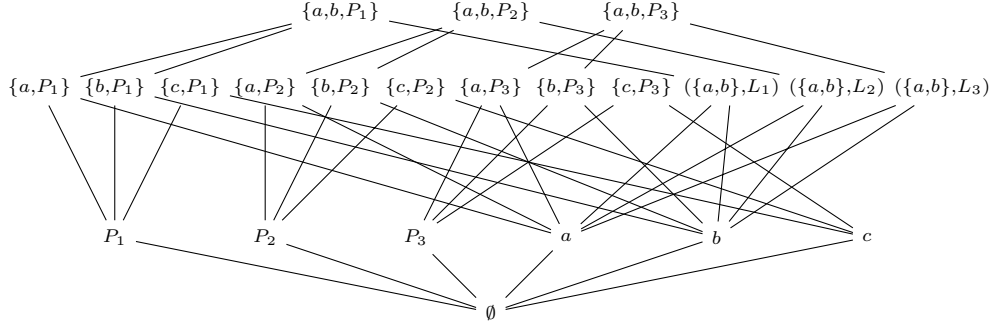

\begin{figure}
    \centering
    \begin{tikzpicture}
        \coordinate (c) at (-1.73,0);
        \coordinate (P1) at (0.1,0.8);
        \coordinate (P2) at (0,0);
        \coordinate (P3) at (0.1,-0.8);
        \coordinate (a) at (1.73,1);
        \coordinate (b) at (1.73,-1);
        
        \fill[fill=red!50, opacity=0.5] (P1) -- (a) to[out=-60,in=60] (b) -- cycle;
        \fill[fill=blue!50, opacity=0.5] (P2) -- (a) -- (b) -- cycle;
        \fill[fill=green!50, opacity=0.5] (P3) -- (a) to[out=-120,in=120] (b) -- cycle;
        
        \draw[opacity=0.5] (c) -- (P1);
        \draw[opacity=0.5] (c) -- (P2);
        \draw[opacity=0.5] (c) -- (P3);

        \draw (P1) -- (a) to[out=-60,in=60] (b) -- cycle;
        \draw (P2) -- (a) -- (b) -- cycle;
        \draw (P3) -- (a) to[out=-120,in=120] (b) -- cycle;
        
        \filldraw[black] (c) circle (1pt) node[anchor=east]{$c$};
        \filldraw[black] (P1) circle (1pt) node[anchor=south east]{$P_1$};
        \filldraw[black] (P2) circle (1pt) node[anchor=south east]{$P_2$};
        \filldraw[black] (P3) circle (1pt) node[anchor=north east]{$P_3$};
        \filldraw[black] (a) circle (1pt) node[anchor=west]{$a$};
        \filldraw[black] (b) circle (1pt) node[anchor=west]{$b$};
    \end{tikzpicture}
    \caption{A realization of the delta complex $n(\LL,\G)$.}
    \label{fig:DeltaComplex}
\end{figure}
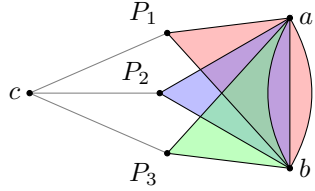

\end{example}

\section{Generalities on Blowups}
\label{sec:geometric_blowups}
We provide a generalization of the lemma by Keel about cohomology of blowups, applicable on the study of the cohomology of wonderful models without the hypothesis of well-connectedness of the building set.

Following \cite{DeConciniProcesi95,DeConciniGaiffi2019}, we define the Chern polynomial 
\[c_E(t)= \sum_{i=0}^d c_i(E)t^{d-i} \in H^*(X)[t],\]
where $c_i(E)$ is the $i$-th Chern class of the vector bundle $E$ of rank $d$ over $X$.
This convention is the opposite of the one used in literature, e.g.\ \cite{FultonIntersectionTheory}.

Let $\iota \colon Z\hookrightarrow Y$ be regular embedding of varieties, and consider the blowup $\Bl_Z(Y)$ of $Y$ along $Z$ and denote with $\pi \colon \Bl_Z(Y)\rightarrow Y$ the projection map. Let $E$ be the exceptional divisor, and $[Z]$ the fundamental class of $Z$ in $H^*(Y)$.
\begin{theorem}[Theorem 1 in the Appendix of \cite{Keel92}]
\label{thm:cohomology:blowup}
    Suppose that the induced map in cohomology $\iota^* \colon H^{\ast}(Y)\rightarrow H^{\ast}(Z)$ is surjective with kernel $J$. Then, the map induced by $\pi^*$ and  defined by $-t\mapsto [E]$ yields an isomorphism
    $$H^{\ast}(\Bl_Z(Y)) \simeq \faktor{H^{\ast}(Y)[t]}{\left(P(t),t\cdot J\right)},$$
    where $P(t) \in H^{\ast}(Y)[t]$ is any polynomial such that:
    \begin{itemize}
    \item its constant term is the cohomology class $[Z]$;
    \item 
 its restriction to $H^{\ast}(Z)$ is the Chern polynomial of $N_Z Y$, the normal bundle to $Z$ in $Y$.
    \end{itemize}
\end{theorem}

We will generalize the above theorem to the case of non-surjective maps $\iota^*$.
Consider the algebra $H^*(Y) \otimes_\ZZ H^*(Z)[t]$ and its subalgebra $A$ generated by $H^*(Y) \otimes 1$ and $1\otimes H^*(Z)t$.
Let $P_z(t) \in A$ be the element 
%\[P(t)=1\otimes t^d + 1 \otimes c_1(N_Z Y) t^{d-1}+ \dots + 1 \otimes c_{d-1}(N_Z Y)t + [Z] \otimes 1 .\]
\begin{equation} \label{eq:Chern_poly}
P_z(t)=1\otimes z t^d + 1 \otimes z c_1(N_Z Y) t^{d-1}+ \dots + 1 \otimes z c_{d-1}(N_Z Y)t + \iota_*(z) \otimes 1 
\end{equation}
for all $z \in H^*(Z)$.
\begin{theorem}\label{thm:blowup}
    Let $\iota\colon Z\hookrightarrow Y$ be a regular embedding, the map 
    %induced by $\pi^*$ obtained by sending $-t$ to $[E]$ 
    defined by $y \otimes 1 \mapsto \pi^*(y)$ and by $1 \otimes zt^k \mapsto -j_*p^*(z) \cdot (-[E])^{k-1}$
    %\pi^*\iota_*(z) \cup (-[E])^{k-1}$
    yields an isomorphism
    \[H^*(\Bl_Z(Y)) \simeq \faktor{A}{(P_z(t), y \otimes t- 1 \otimes \iota^*(y)t)}\]
\end{theorem}
\begin{proof}
    The following proof is analogous to the one in \cite[Appendix, Theorem 1]{Keel92}, but in our case the map $\iota^*$ not need to be surjective.
    
    Let $E$ be the exceptional divisor, from the exact sequence
    \[ 0 \to H^*(Y) \to H^*(\Bl_Z Y) \oplus H^*(Z) \to H^*(E) \to 0\]
    we have
    \begin{equation} \label{eq:additive_cohomology}
        H^{\ast}(\Bl_Z(Y)) \simeq H^*(Y) \oplus H^*(Z)t \oplus \dots \oplus H^*(Z)t^{d-1} \simeq \faktor{A}{(P_z(t), y \otimes t- 1 \otimes \iota^*(y)t)}
    \end{equation}
    as modules and the map $A \to H^{\ast}(\Bl_Z(Y))$ is surjective. 
    Therefore, it is enough to verify that the relations $P_z(t)$ and $y \otimes t- 1 \otimes \iota^*(y)t$ hold in $H^*(\Bl_Z(Y))$ for any $z \in H^*(Z)$ and any $y \in H^*(Y)$.
    Consider the diagram
\begin{center}
\begin{tikzcd}
E \arrow[d, "p"'] \arrow[r, "j", hook] & \Bl_Z(Y) \arrow[d, "\pi"] \\
Z \arrow[r, "\iota", hook]                   & Y 
\end{tikzcd}
\end{center}
    and using the projection formula we have:
    \[ y \otimes t \mapsto \pi^*(y) \cdot (-[E]) = -\pi^*(y) \cdot j_*(1) = -j_*(j^* \pi^* (y) \cdot 1), \]
    \[ 1 \otimes \iota^*(y) t \mapsto -j_*(p^* \iota^*(y)) = -j_*(j^* \pi^* (y)). \]
    Now, consider the universal quotient bundle $F= p^*N_Z Y/\mathcal{O}_E(-1)$. By multiplicativity of the Chern polynomial on short exact sequences the equality
    $ p^*c_{N_Z Y} = c_{p^* N_Z Y} = c_{\mathcal{O}_E(-1)} c_F $ holds,
    where $c_{\mathcal{O}_E(-1)}(t) = t+j^*[E]$. 
    Hence $c_i(p^*N_Z Y) = c_i(F) + c_{i-1}(F) j^*[E]$.
    Solving the recurrence relations we obtain 
    \[c_{d-1}(F) = \sum_{k=1}^d p^*c_k(N_Z Y) \cdot j^*(-[E])^{d-k-1}.\]

    The element $P_z(t)$ is mapped to zero because:
    \begin{align*}
        P_z(t) \mapsto &\sum_{k=0}^{d-1} -j_*p^*(z \cdot c_k(N_Z Y)) \cdot (-[E])^{d-k-1} + \pi^*\iota_*(z) \\
        &= -\sum_{k=0}^{d-1} j_* (p^*z \cdot p^*c_k(N_Z Y) \cdot j^*(-[E])^{d-k-1}) + \pi^*\iota_*(z) \\
        &=  -j_* (p^*z \cdot \sum_{k=0}^{d-1} p^*c_k(N_Z Y) \cdot j^*(-[E])^{d-k-1}) + \pi^*\iota_*(z) \\
        &= -j_* (p^*z \cdot c_{d-1}(F)) + \pi^*\iota_*(z) =0.
    \end{align*}
    The last equality follows from the excess intersection theorem
    \[ j_* (p^*z \cdot c_{d-1}(F)) = \pi^*\iota_*(z). \]
    This completes the proof.
\end{proof}

\begin{lemma} \label{lemma:normal_bundle}
    Let $W,Z$ be smooth subvarieties in $Y$ and $\pi\colon \tilde{Y}=\Bl_Z(Y) \to Y$ be the canonical projection. Let $\tilde{W}$ be the dominant transform of $W$ and $E$ be the exceptional divisor.
    Then:
    \begin{enumerate}
        \item If $Z \subset W$ then $c_{N_{\tilde{W}}\tilde{Y}}(t) = \pi^*c_{N_{W} Y}(t-[E])$.
        \item If the intersection of $Z$ and $W$ is transversal then $c_{N_{\tilde{W}}\tilde{Y}}(t) = \pi^* c_{N_{W} Y}(t)$.
    \end{enumerate}
\end{lemma}

\begin{proof}
    The latter statement follows from $N_{\tilde{W}}\tilde{Y}\simeq \pi^* N_{W} Y$, see \cite[Lemma 7.1]{DeConciniGaiffi2019}.
    The first one follows by $N_{\tilde{W}}\tilde{Y}\simeq \pi^* N_{W} Y \otimes \mathcal{O}(-E)$ and by $c_{N \otimes L}(t)=c_N(c_L(t))$ for any vector bundle $N$ and any line bundle $L$, see e.g.\ \cite[Remark 3.2.3 b)]{FultonIntersectionTheory}.
\end{proof}

\section{Recollection on toric wonderful models}
\label{sec:toric_wonderful_model}

In this section we recall the contruction of certain toric varieties from \cite{DeConciniGaiffi2018} which will be used as compactification for toric wonderful models. We then quickly recall the construction of wonderful models for arrangements of subvarieties in the sense of \cite{Li09}.
Finally, we provide a suitable compactification for the running example.\\

Consider a toric arrangement $\AAA$ in the ambient torus $T$. Fix a basis of the character group $X^{\ast}(T)$, and get isomorphisms $X^{\ast}(T)\simeq\ZZ^n$, $X_{\ast}(T)=\Hom(X^{\ast}(T),\ZZ) \cong \ZZ^n$, and $T\cong(\CC^{\ast})^n$. It will be useful to recall some more notation: given a layer $\KK_{\Gamma,\phi}$ of $\AAA$, define:
\begin{itemize}
    \item $V := \Hom_\ZZ(X^{\ast}(T),\RR) = X_{\ast}(T)\otimes_\ZZ \RR$
    \item $\Ann(\Gamma):= \{v\in V \mid \left\langle\chi,v\right\rangle=0\text{, for all }\chi\in\Gamma \}$
\end{itemize}
In \cite{DeConciniGaiffi2018}, De Concini and Gaiffi introduced a wonderful compactification for the complement space of a toric arrangement. They embedded the ambient torus $T$ in a \emph{good} toric variety $X_{\Delta}$ where $\Delta\subseteq V$ with respect to the arrangement $\AAA$. We now recall the definition of good toric variety, and its main properties.
\begin{definition}
Consider a basis $\{\chi_1,\ldots,\chi_t\}$ of $\Gamma$, and any cone $C=\langle r_1,\ldots, r_k\rangle_{\geq 0}$ of $\Delta$. We say that the chosen basis has the \emph{equal-sign property} with respect to $C$ if $\langle \chi_i,r_j\rangle \geq 0$ or $\langle \chi_i,r_j\rangle \leq 0$ for all $i=1,\ldots,t$ and $j=1,\ldots k$. If every layer $\KK_{\Gamma,\phi}$ of $\AAA$ has has an equal-sign basis with respect to every cone of $\Delta$, we say that the toric variety $X_\Delta$ is \emph{equal-sign} with respect to $\AAA$. If $X_\Delta$ is also a smooth, projective variety, we say that it is a \emph{good} toric variety for $\AAA$.
\end{definition}
De Concini and Gaiffi provided an algorithm in \cite{DeConciniGaiffi2018} that, given a complete, smooth toric variety $X_\Delta$ and a toric arrangement, constructs a good toric variety for $\AAA$ coming from a subdivision of the fan $\Delta$. Every subdivision of the fan corresponds geometrically to a blowup of the closure of a toric orbit of codimension $2$.

\begin{theorem}
    Let $X_\Delta$ be a good toric variety for $\AAA$. Let $\KK_{\Gamma,\phi}$ be a layer, and $T' = \bigcap_{\chi\in\Gamma}\ker(\chi)$ the subtorus of $T$ associated to it. Then, the following facts hold:
    \begin{enumerate}
        \item For every cone $C\in\Delta$, its relative interior is either entirely contained in $\Ann(\Gamma)$, or disjoint from it.
    \item The set of cones of $\Delta$ contained in $\Ann(\Gamma)$ form a smooth fan with open orbit $T'$
    \item The closure of the layer $\overline{\KK_{\Gamma,\phi}}$ in $X_\Delta$ is a smooth toric variety whose fan is $\Delta_{\KK_{\Gamma,\phi}}:=\Delta\cap\Ann(\Gamma).$
    \item Let $\mathcal{O}$ be a $T$-orbit of $X_\Delta$, and $C_{\mathcal{O}}$ be its corresponding cone in $\Delta$. Then: \begin{itemize} \item if $C_{\mathcal{O}}\cap\Ann(\Gamma) = \emptyset$, then $\bar{\mathcal{O}}\cap\overline{\KK_{\Gamma,\phi}} = \emptyset$
    \item if $C_{\mathcal{O}}\subseteq\Ann(\Gamma)$, then $\mathcal{O}\cap\overline{\KK_{\Gamma,\phi}}$ is the $T'$-orbit in $\overline{\KK_{\Gamma,\phi}}$ corresponding to the cone $C_{\mathcal{O}}$
    \end{itemize}
    \end{enumerate}
\end{theorem}

Essentially, a good toric variety is smooth, projective, and such that the intersection of the closures of the layers in $\AAA$ and of the the toric divisors intersect nicely. That is, such set of intersections forms an \emph{arrangement of subvarieties} in the sense of Li (see \cite{Li09}). We recall some of Li's definitions:

\begin{definition}
    A \emph{simple arrangement of subvarieties} of a nonsingular variety $X$ is a finite set $\Lambda$ of nonsingular, closed, connected subvarieties properly contained in $X$ such that:
    \begin{itemize}
        \item for all $Z_1,Z_2\in\Lambda$, $Z_1\cap Z_2$ is either empty, or an element of $\Lambda$

        \item Whenever $Z_1\cap Z_2\neq\emptyset$, the intersection is \emph{clean}, i.e. it is nonsingular, and for all $z\in Z_1 \cap Z_2$ we have the condition on tangent spaces: $$\T_z(Z_1\cap Z_2) = \T_z(Z_1)\cap\T_z(Z_2)$$
    \end{itemize}
\end{definition}

\begin{definition}
    An \emph{arrangement of subvarieties} of a nonsingular variety $X$ is a finite set $\Lambda$ of nonsingular, closed, connected subvarieties properly contained in $X$ such that:
    \begin{itemize}
        \item for all $Z_1,Z_2\in\Lambda$, $Z_1\cap Z_2$ is either empty, or a disjoint union of elements of $\Lambda$

        \item Whenever $Z_1\cap Z_2\neq\emptyset$, the intersection is clean.
    \end{itemize}
\end{definition}
In our case, $X=X_\Delta$ is a good toric variety for $\AAA$, and $\Lambda$ is the aforementioned set of connected components of intersections of $$\{\overline{\KK}\}_{\KK\in\AAA}\cup \{D\mid D\text{ irreducible component of }X_\Delta\setminus T\}.$$ In the last section of the paper, we will analyze some of these intersections after the blowup process. This will be the first step for the definition of a Morgan model for the arrangement. \\ A wonderful compactification in the sense of Li for the complement space can be constructed by blowing up a in a prescribed order a certain subset $\G\subseteq\Lambda$ of varieties, a \emph{building set}.
\begin{definition}
    Let $\Lambda$ be a simple arrangement of subvarieties in $X$. A subset $\G\subseteq\Lambda$ is a \emph{building set} for $\Lambda$, if for each subvariety $Z\in\Lambda$, the minimal (with respect to inclusion) elements of $\{G\in\G\mid G\supseteq Z\}$ intersect transversally, and their intersection is $Z$. The minimal elements are called the $\G$-factors of $Z$.
\end{definition}
Let $U\subseteq X$ be an open set. Denote with $\Lambda_U$ the arrangement $$\Lambda_U := \{Z\cap U\mid Z\in\Lambda, Z\cap U\neq\emptyset \}.$$
\begin{definition}
    Let $\Lambda$ be an arrangement of subvarieties in $X$. A subset $\G\subseteq\Lambda$ is a \emph{building set} for $\Lambda$, if there is an open cover $\mathcal{U}$ of $X$ such that:
    \begin{itemize}
        \item the arrangement $\Lambda_U$ is simple, for every $U\in\mathcal{U}$
        \item the set $\G_U$ is a building set for $\Lambda_U$, for every $U\in\mathcal{U}$
    \end{itemize}
\end{definition}

One can check that the definition of building sets (and nested sets) in the previous section are the combinatorial analogue of those given by Li. For instance, transversality of intersections of elements of $\G$ is expressed in term of the poset of layers by the condition of \Cref{buildingsetdef}.

\begin{definition}[\cite{Li09}, Definition 1.1]
    Given an arrangement of subvarieties $\Lambda$ in a nonsingular variety $X$ and a building set $\G$, a \emph{wonderful compactification} $Y_\G$ for the complement space $X\setminus\bigcup_{Z\in\Lambda}Z$ is given by the closure of the image of the locally closed embedding: $$X\setminus\bigcup_{Z\in\Lambda}Z \hookrightarrow \prod_{G\in \G}\Bl_{G}X.$$
\end{definition}

\begin{theorem}[\cite{Li09}, Theorem 1.3]
    Let $X$ be nonsingular variety, and $\G = \{G_1\ldots,G_m\}$ building set ordered in such a way that $\G_k = \{G_1,\ldots, G_k\}$ is building set for every $k=1,\ldots,m$. Then, $$Y_\G = \Bl_{\widetilde{G}_m}\cdots\Bl_{\widetilde{G}_2}\Bl_{G_1} X.$$
\end{theorem}

Again, the combinatorial blow-up in the previous section acts on the poset of layers in the same way as the geometric blowup acts on $X$. We now choose a good toric variety for our running example, after a couple of remarks about toric varieties.

\begin{theorem}\label{thm:cohomology_toric}
The cohomology ring of a toric variety $X_{\Delta}$ is given by
$$ H^{\ast}(X_{\Delta},\ZZ) = \mathbb{Z}[c_r]_{r\in\mathcal{R}}/L_\Delta ,$$
where $\mathcal{R}$ is the set of rays of $\Delta$, and the ideal $L_\Delta$ is generated by
\begin{itemize}
\item[a)]$c_{r_1}\cdots \ c_{r_k}$, for all subsets of rays $\{c_{r_1},\ldots,c_{r_k}\}$ not spanning a cone;
\item[b)]$\sum_{r\in\mathcal{R}}\langle\beta,r\rangle c_r$, for any $\beta\in X^{\ast}(T)$.
\end{itemize}
\end{theorem}

\begin{comment}
This result, together with \cite[Theorem 5.1]{DeConciniGaiffi2019} implies:
\begin{proposition}\label{prop:restriction_cohomology}
Let $X_\Delta$ be a good toric variety for the arrangement $\AAA$ and $\KK_{\Gamma,\phi}\in\AAA$.
The restriction map in cohomology 
$$j^{\ast} \colon H^{\ast}(X_\Delta,\ZZ)\rightarrow H^{\ast}(\overline{\KK}_{\Gamma,\phi})$$
is surjective with kernel generated by the classes $c_r$ with $r\in\mathcal{R}$ such that $r\notin \Ann(\Gamma)$.
\end{proposition}

\end{comment}

\begin{remark}
A toric variety $X_{\Delta}$ is smooth if:
    \begin{itemize}
        \item Every $l$-dimensional cone in $\Delta$ is be simplicial, i.e.\ generated by $l$ rays;
        \item For every cone $C \in \Delta$ the primitive vectors of the rays can be completed to a basis for the lattice $\Gamma$. 
    \end{itemize}
\end{remark}
For general information on toric varieties, we refer to \cite{FultonToricVarieties}.

\begin{center}
\begin{figure}
    \includegraphics[width = 50mm]{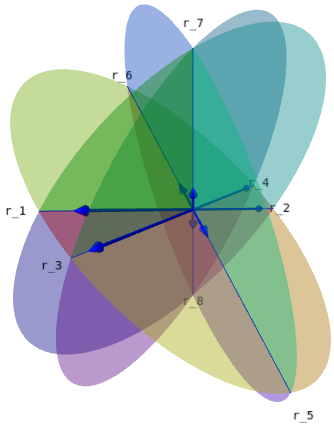}
    \includegraphics[width = 60mm]{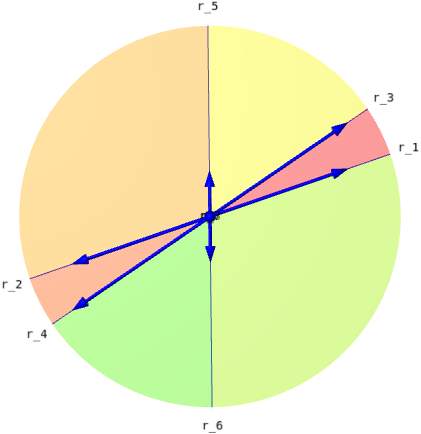}
    \caption{Equal-sign, non-smooth toric variety for $\AAA$ and its restriction to the hyperplane $x-z=0$.}
\label{fig:RunningFan}
\end{figure}
\end{center}

\begin{example}
\label{running_fan}
    Consider the arrangement $\AAA$ of \Cref{running}.
    We choose a good toric variety in the following way. In Section 8 of \cite{DeConciniGaiffi2018} it is shown how to choose an equal-sign, but not necessarily smooth, toric variety for an arrangement. 
    %We now see how.
    Consider the hyperplane arrangement defined by $\AAA$:
    \begin{align*}
        & a \colon x=0,\\
        & b \colon x-3y = 0,\\
        & c_1 \colon x-z = 0,\\
        & c_2 \colon 2x-3y = 0,
    \end{align*}
    and the complete fan $\Delta$ defined by the closure of the chambers of this arrangement.
    We depict $\Delta$ in \Cref{fig:RunningFan}.
    Notice that all the maximal cones contain one of the rays $r_7=(0,0,1)$ or $r_8=(0,0,-1)$, hence the interesting part of the fan is contained in the plane $x-z=0$; we depict the intersection of the fan with this plane.
    The fan $\Delta$ obtained in this way has the equal-sign property with respect to $\AAA$ (by \cite[Section 8]{DeConciniGaiffi2018}).
    In this case $\Delta$ is not smooth, as the rays $r_3 = (3,2,3)$, $r_5=(0,1,0)$ spanning a cone cannot be completed to a $\ZZ$-basis. 
    %Indeed, this would only be possible by adding the rays $r_7$ or $r_8$, for which: $$\left|\det
  %\begin{bmatrix}
  %  3 & 0 & 0\\
  %  2 & 1 & 0\\
  %  3 & 0 & \pm 1\\
  %\end{bmatrix}\right| \neq 1.$$
  To get a good toric variety, it is enough to resolve the singularities of $X_\Delta$. In this case, we subdivide only $2$-dimensional cones on the plane $x-z=0$.
  A resolution of singularities is given by the fan $\tilde{\Delta}$ with rays: 
%For later reference, we list the set of rays of this fan:
\begin{center}
\begin{tabular}{ l l }
 $r_1 = (3,1,3)$ & $r_2 = (-3,-1,-3)$  \\ 
 $r_3 = (3,2,3)$ & $r_4 = (-3,-2,-3)$  \\ 
 $r_5 = (0,1,0)$ & $r_6 = (0,-1,0)$  \\ 
 $r_7 = (0,0,1)$ & $r_8 = (0,0,-1)$  \\ 
 $r_9 = (1,1,1)$ & $r_{10} = (2,1,2)$  \\ 
 $r_{11} = (1,0,1)$ & $r_{12} = (-1,-1,-1)$  \\ 
 $r_{13} = (-2,-1,-2)$ & $r_{14} = (-1,0,-1).$  
\end{tabular}
\end{center}
Like for $\Delta$, the maximal cones are given by $\{r_i,r_j,r_7\}$ and $\{r_i,r_j,r_8\}$ where $\{r_i,r_j\}$ is a maximal cone in the intersection $\tilde{\Delta} \cap \{x- z=0 \}$ depicted in \Cref{fig:SmoothRunningFan}.
%  We depict the intersection of the plane with the resulting toric fan in \Cref{fig:SmoothRunningFan}.
\end{example}

\begin{figure}
    \includegraphics[width = 75mm]{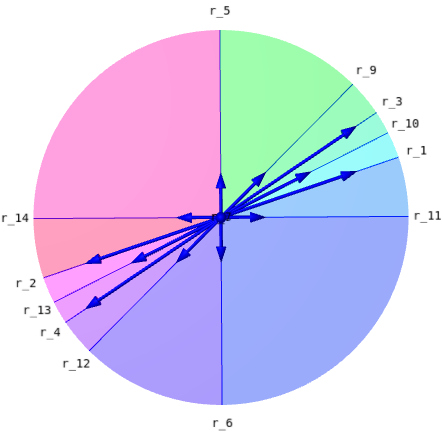}
    \caption{Equal-sign and smooth toric variety for $\AAA$, restricted to the hyperplane $x-z=0$.}
\label{fig:SmoothRunningFan}
\end{figure}

\section{Presentation of the cohomology ring}
\label{sec:presentation}
We use the same notations of the previous sections: $X_{\Delta}$ is a good toric variety for the arrangement $\AAA$, $\LL=\LL(\AAA)$ denotes its poset of layers, and $B=H^{\ast}(X_\Delta,\ZZ)$ is its cohomology ring.
%Let $Z \in \LL$ and consider the closure of the corresponding toric variety $\overline{\KK}_Z=\overline{\KK}_{\Gamma_Z,\phi_Z}$in $X_\Delta$, so that $\Gamma_Z$ is the split direct summand in the character group. 
%Choose any polynomial $P_Z=P_Z^{\hat0}$ in the ring $B[t]$ satisfying the following properties:
%\begin{itemize}
%    \item the constant term of $P_Z$ is the cohomology class $[\overline{\KK}_Z]\in B$
%    \item the restriction map to $H^{\ast}(\overline{\KK}_Z)$ sends $P_Z$ to the Chern polynomial of $N_{\overline{\KK}_Z}X$, the normal bundle to the toric variety $\overline{\KK}_Z$ in $X$.
%\end{itemize}
%Following \cite{DeConciniGaiffi2019}, we call such a polynomial a \emph{good lifting} of the Chern polynomial of the normal bundle $N_{\overline{\KK}_Z}X$.
%Similarly, for the pair $(Z,W)$ in $\LL$ with $W \leq Z$ fix a basis $(\chi_1,\ldots , \chi_s)$ of $\Gamma_Z$ such that $(\chi_1,\ldots,\chi_k)$ is a basis for $\Gamma_{W}$ for $k = \cd(W)\leq \cd(Z) = s$; consider a good lifting of the Chern polynomial of the bundle $N_{\overline{\KK}_Z}(\overline{\KK}_{W})$, and denote it $\overline{P}_Z^{W}\in H^{\ast}(\overline{\KK}_{W})$. Finally, define $P_Z^{W}$ to be a lifting of the previous polynomial to $B[t]$.

\begin{remark}
    In Section 8 of \cite{DeConciniGaiffi2019}, and in Section 4 of \cite{PagariaMoci}, the authors exhibited the following possible explicit choices for the lift of Chern polynomials:
    \[P_Z^{W} (t) = \prod_{j=k+1}^{s}\left( t - \sum_{r\in\mathcal{R}} \min\left(0, \langle\chi_j, r\rangle\right) c_r \right) \in H^{\ast}(X_\Delta,\ZZ)[t]\]
    and respectively
    \[P_Z^{W} (t) = t^{s-k} + \prod_{j=k+1}^{s}\left(- \sum_{r\in\mathcal{R}} \min\left(0, \langle\chi_j, r\rangle\right) c_r \right) \in H^{\ast}(X_\Delta,\ZZ)[t].\]
    Our approach is similar to the second choice for the lift.
\end{remark}

Consider the graded ring $B[t_a \mid a \in (\Bl_\G \LL) \setminus \{\hat{0}\}]$ with $\deg t_a= 2\rk(a)$.
For the sake of notation we set $\tau_F = -\sum_{G \in \mathcal{G}_{\geq F}} t_{\{G\}}$ for any $F \in \G$.

\begin{definition}
\label{def:cohomologyring}
    Let $R(X_\Delta,\G)$ be the ring $B[t_a \mid a \in (\Bl_\G \LL) \setminus \{\hat{0}\}]$ with $\deg t_a= 2\rk(a)$ modulo the ideal $I_{\mathcal{G}}$ generated by the following elements:
\begin{enumerate}[label=\roman*)]
    \item $c_rt_a$ for all $a \in \Bl_\G(\LL)$ and all $r \not \in \Ann \Gamma_{\pi(a)}$.
    \item For all $a \lessdot b \in \Bl_\G \LL$ such that $\pi(a) < \pi (b) $ set $G\in \G$ the unique element such that $b \in \{ G \} \vee a$. The following element
    \begin{equation}
    \label{eq:rel_chern_poly}
        \Bigl(-\sum_{\substack{c \gtrdot a \\ \pi(c) \geq \pi(b)}} t_c \Bigr) \tau_G^{s-1} +t_a \prod_{j=1}^{s}\left(- \sum_{r\in\mathcal{R}} \min\left(0, \langle\chi_j, r\rangle\right) c_r \right)
    \end{equation}
    where $\chi_1, \dots, \chi_s \in \Gamma_{\pi(b)}$ are such that their image in $\Gamma_{\pi(b)}/\Gamma_{\pi(a)}$ form a basis of this lattice.
    \item $t_at_b - t_{a \wedge b} \sum_{c \in a \vee b} t_c$ for all $a,b \in (\Bl_\G \LL) \setminus \{\hat{0}\}$.
\end{enumerate}
\end{definition}

\begin{remark}
    Relation (3) implies that $t_at_b=0$ if $a \vee b= \emptyset$. Moreover, in the case $a \vee b \neq \emptyset$, the element $a \wedge b$ exists and is unique.
\end{remark}

Let us denote relation \eqref{eq:rel_chern_poly} by $Q^b_a(\G)$.
Consider the total order $\succ_{\G}$ on $\G$, with last element $G\in \G$ and associated nested set $g \in \Bl_\G(\LL)$.
Let $\iota\colon \overline{\mathcal{K}}_G' \hookrightarrow Y(X,\G')$ be the embedding of the closure of the layer $\mathcal{K}_G' \subset T$ in the wonderful model for the deleted arrangement $Y(X,\G')$.
Recall from \cite[Theorem 3.1]{DeConciniGaiffi2018} that $\overline{\mathcal{K}}_G'\simeq Y(X_{\Delta_G},\G'')$.\\

\begin{definition}
Let $a\in\Bl_\G(\LL)$. By \Cref{remark:BlG}, to $a$ is associated to a nested set $(S,x)$ where $S\subset \G$ and $x\in \vee S $. Let $D_a$ be the strata in $Y(X_\Delta,\G)$ corresponding to the nested set $a$.
\end{definition}

Our main theorem is the following:
\begin{theorem} \label{thm:main}
The cohomology ring $H^*(Y(X_\Delta,\G);\ZZ)$ is isomorphic to $R(X_\Delta,\G)$ by sending $t_a$ to the cohomology class $[D_a]$.
\end{theorem}

\begin{lemma}
\label{lemma:pullback}
    In the setting above, we have
    \begin{align*}
        \iota^*(c_r)=  \begin{cases} c_r & \textnormal{if } r \in \Ann(\Gamma_{G}) \\ 0, & \textnormal{otherwise,} \end{cases} \quad \quad \iota^*(t_a) =  \sum_{c \in a \vee g} t_c.
    \end{align*}
\end{lemma}
\begin{proof}
    The first statement can be checked before the sequence of blowups prescribed by $\G'$ and it follows from the general theory of toric varieties applied to the map $X_{\Delta_G} \hookrightarrow X_\Delta$, c.f.\ \cite[Proposition 5.3]{DeConciniGaiffi2019}.
    For the second claim, if $a \vee g \neq \emptyset$ then the subvarieties $D_a$ and $\overline{\mathcal{K}}_G'$ intersect transversally and so 
    \[ \iota^*(t_a) = -\iota^*([D_a]) = -[D_a \cap \overline{\mathcal{K}}_G'] = \sum_{c \in a \vee g} t_c.\]
    Otherwise $a \vee g = \emptyset$, hence the intersection $D_a \cap \overline{\mathcal{K}}_G'$ is empty in $Y(X,\G')$. Therefore, $\iota^*(t_a) =0$.
\end{proof}

\begin{lemma} \label{lemma:pushforward}
Let $\iota$ be the embedding of $\overline{\mathcal{K}}_G' \simeq Y(X_{\Delta_G},\G'')$ in $Y(X,\G')$ and $b$ be a nested set such that $b \geq g \in \Bl_{\G}(\LL)$.
Then:
    \[\iota_*(t_b) = 
    %\Bigl( \sum_{\substack{c \gtrdot a\\ \pi(c) >\pi(b)}} t_c\Bigr) 
    t_a 
    \Bigl( \sum_{H \in \G_{>G}} t_{ \{H\}} \Bigr)^d\] 
    if $ G \not \in F(\G,b)$ and 
     \[\iota_*(t_b) = 
    \Bigl( \sum_{\substack{ c \gtrdot a\\ \pi(c) >\pi(b)}} t_c\Bigr) \Bigl( \sum_{H \in \G_{>G}} t_{ \{H\}} \Bigr)^{d-1} + t_a \prod_{j=1}^{d}
    \left(
    %\left( \sum_{H \in \G_{>\pi(b)}} t_{\{H\}} \right)
    - \sum_{r\in\mathcal{R}} \min\left(0, \langle\chi_j, r\rangle\right) c_r \right)\]
    if $ G \in F(\G,b)$ where $a\in \Bl_{\G'}(\LL)$ is such that $b \in a \vee g$ and  $\{\chi_i\}_{i=1, \dots, d}$ is a basis of $\Gamma_{\pi(b)}/\Gamma_{\pi(a)}$.
\end{lemma}
\begin{proof}
Firstly, consider the case $\iota_*(1)$ (i.e.\ $b=g$).
The fundamental class of $\overline{\mathcal{K}}_G \subset X$ is 
\[
\prod_{j=1}^{d}\left(- \sum_{r\in\mathcal{R}} \min\left(0, \langle\chi_j, r\rangle\right) c_r \right)
\]
where $\{\chi_i\}_i$ is a basis of $\Gamma_{\pi(g)}$, see \cite[eq.~(2)]{DeConciniGaiffi2019}.
By classical results on blowups we have
\[ \iota_*(1) = \prod_{j=1}^{d}\left(\Bigl( \sum_{H \in \G_{>G}} -[D_H] \Bigr) - \sum_{r\in\mathcal{R}} \min\left(0, \langle\chi_j, r\rangle\right) c_r \right). \]
Recall that $t_{\{H\}}$ is defined as $-[D_H]$ and that $t_hc_r \min(0,\langle \chi_j,r\rangle)=0$ for any $h=\{H\}$ with $H \geq G$ and any $r \in \mathcal{R}$; this completes the case $b=g$.

Now, suppose that $\pi(a) > G$; in this case $t_b=\iota^*(t_a)$ because $b=a \vee g$.
It follows that
\[ \iota_*(t_b)= \iota_*(\iota^*(t_a)) = \iota_*(1) t_a = t_a 
    \Bigl( \sum_{H \in \G_{>G}} t_{ \{H\}} \Bigr)^d \]
because $\Ann \Gamma_{\pi(a)} \subset \Ann \Gamma_G$ and so $\min\left(0, \langle\chi_j, r\rangle\right) c_r t_a=0$ for all $r\in \mathcal{R}$ and all $j=1, \dots, d$. 

    Finally, consider the following commutative diagram
    \begin{center}
    \begin{tikzcd}
D''_b \arrow[d, "f"'] \arrow[r, "\kappa"] & D'_a \arrow[d, "h"] \\
\overline{\mathcal{K}}_G' \arrow[r, "\iota"]            & {Y(X,\G')}        
    \end{tikzcd}       
    \end{center}
    where all maps are inclusions.
    By the property of blowup we have $D_a' \simeq Y(X_{\Delta_{\pi(a)}},\G_{\pi(a)})$.
    Notice that the codimension of $D''_b$ in $D'_a$ is exactly $\rk(G)$. Moreover, since $b$ is $\G$-nested and $\pi(a) \not \geq G$ it follows that $\Gamma_{\pi(b)}/\Gamma_{\pi(a)}$ has rank $d$.
    We have
    \[\iota_*(t_b) =\iota_*(-f_*(1)) = -h_* (\kappa_*(1))\]
    and by the first case
    \begin{align*}
        \kappa_*(1)&= \prod_{j=1}^{d} \Bigl( \Bigl( \sum_{\substack{c \gtrdot a \\ \pi(c)> \pi(b)}} t_{c} \Bigr)- \sum_{r\in\mathcal{R}_a} \min\left(0, \langle\chi_j, r\rangle\right) c_r \Bigr) \\
        &= \Bigl( \sum_{\substack{c \gtrdot a \\ \pi(c)> \pi(b)}} t_{c} \Bigr)^d + \prod_{j=1}^{d} \sum_{r\in\mathcal{R}} -\min\left(0, \langle\chi_j, r\rangle\right) c_r 
    \end{align*}
    where $\{\chi_i\}_{i=1, \dots, s}$ form a basis of $\Gamma_{\pi(b)}/\Gamma_{\pi(a)}$.

    We have
    \begin{align*}
        h_* \Bigl( \prod_{j=1}^{d} \sum_{r\in\mathcal{R}_a} -\min\left(0, \langle\chi_j, r\rangle\right)  c_r \Bigr) &= h_* \Bigl( h^* \Bigl(\prod_{j=1}^{d}  \sum_{r\in\mathcal{R}} -\min\left(0, \langle\chi_j, r\rangle\right) c_r \Bigr) \Bigr) \\
        &= t_a \prod_{j=1}^{d}  \sum_{r\in\mathcal{R}} -\min\left(0, \langle\chi_j, r\rangle\right) c_r
    \end{align*} 
    where by abuse of notation we write $c_r$ both for the element in $H^*(D'_a)$ and the element in $H^*(Y(X,\G'))$ because $h^*(c_r)=c_r$.

    On the other hand we have
    \begin{align*}
        h_* \Bigl( \Bigl( \sum_{\substack{c \gtrdot a \\ \pi(c)> \pi(b)}} t_{c} \Bigr)^d \Bigr) &= \sum_{\substack{e \gtrdot a \\ \pi(e)> \pi(b)}} h_* \Bigl(t_e \Bigl( \sum_{\substack{c \gtrdot a \\ \pi(c)> \pi(b)}} t_{c} \Bigr)^{d-1} \Bigr) \\
        &= \sum_{\substack{e \gtrdot a \\ \pi(e)> \pi(b)}} h_* \Bigl(t_e \Bigl( \sum_{\substack{c \gtrdot a \\ \pi(c)> G \\ c \not \in a \vee g}} t_{c} \Bigr)^{d-1} \Bigr) \\
        &= \sum_{\substack{e \gtrdot a \\ \pi(e)> \pi(b)}} h_* \Bigl(t_e h^*\Bigl( \sum_{H \in \G_{>G}} t_{\{H\}} \Bigr)^{d-1} \Bigr) \\
        &= \Bigl( \sum_{H \in \G_{>G}} t_{\{H\}} \Bigr)^{d-1}  \sum_{\substack{e \gtrdot a \\ \pi(e)> \pi(b)}} h_* (t_e ) \\
        &= \Bigl( \sum_{H \in \G_{>G}} t_{\{H\}} \Bigr)^{d-1}  \sum_{\substack{e \gtrdot a \\ \pi(e)> \pi(b)}} t_e 
    \end{align*}
This completes the proof.
%    \begin{align*}
%        \iota_*(t_b) &=\iota_*(-f_*(1)) = -h_* (g_*(1)) \\
%        &= -h_* \Bigl( \prod_{j=1}^{s} \Bigl( \Bigl( \sum_{H \in \G_{>\pi(b)}} -t_{\{H\}} \Bigr)- \sum_{r\in\mathcal{R}} \min\left(0, \langle\chi_j, r\rangle\right) c_r \Bigr) \Bigr) \\
%        &= -h_* (h^*(\prod_{j=1}^{s}\left(- \sum_{r\in\mathcal{R}} \min\left(0, \langle\chi_j, r\rangle\right) c_r \right))) \\
%        &= -h_*(1) \prod_{j=1}^{s}\left(- \sum_{r\in\mathcal{R}} \min\left(0, \langle\chi_j, r\rangle\right) c_r \right) \\
%        &= t_a \prod_{j=1}^{s}\left(- \sum_{r\in\mathcal{R}} \min\left(0, \langle\chi_j, r\rangle\right) c_r \right) 
%    \end{align*}
%    where we used \cite[eq.~(2)]{DeConciniGaiffi2019} to make $g_*(1)$ explicit and the identity $h^*(c_r)=c_r$ for any $r \in \Ann(\Gamma_a)$.
\end{proof}

For our embedding $\iota \colon \overline{\mathcal{K}}_G' \hookrightarrow Y(X,\G')$ the polynomial $P_z(t)$ of eq.\ \eqref{eq:Chern_poly} using \Cref{lemma:normal_bundle} specializes to
\begin{equation}  \label{eq:Chern_poly_Y}
P_z(t)= \sum_{i=1}^d 1\otimes z \binom{d}{i} \Bigl( \sum_{H \in \G_{>G}} t_h \Bigr)^{d-i} t^{i}  + \iota_*(z) \otimes 1 
\end{equation}
where $\iota_*(z)$ is computed in \Cref{lemma:pushforward} in the case $z=t_b$.\\
%(Fare anche il caso $z$ generico con push-pull?) \todo{cosa fare?}

%Now we fix a total order on $\G=\{G_1, \dots, G_m\}$ and we define $\G' = \G \setminus \{G_m \} $, $\LL' \subset \LL$ be the poset of flats of $\G'$, $\G''$ as in \Cref{lemma:building_in_contr} and $\LL''= \LL_{\geq G_m}$.
%Notice that $\G'$ and $\G''$ are building in $\LL'$ and $\LL''$ respectively.

%With this in mind we define $D_a=(\cap_{G\in S}D_{G})\cap D_{x}$, where $D_{G}$ is the strict transform of the exceptional divisor for the blowup at $\overline{\mathcal{K}}_{G}$.

\section{An additive basis}
\label{sec:additive_basis}

%A monomial basis for the cohomology ring of smooth, projective toric varieties is given in Chapter 5 of \cite{FultonToricVarieties}.
%Our good toric varieties satisfy these properties, so we have a monomial basis for $H^*(X;\mathbb{Z})$ [...] \todo{questo periodo è fuori contesto}

In this section we describe a monomial basis of $R(X,\mathcal{G})=H^*(Y(X,\mathcal{G}),\mathbb{Z})$ (\Cref{thm:Grobner_Basis}).
A similar result was proven by \cite{GaiffiPapiniSiconolfi2022} with the further hypothesis of $\mathcal{G}$ being  well-connected. Here we introduce two main ingredients which are the adaptation of the definition of ``admissible function'' to our setting and to the description of a Gr\"obner basis for the ideal $I_{\G}+L_{\Delta}$. The conclusion of the proof of \Cref{thm:Grobner_Basis} can be found in next Section (\ref{sec:main_proof}).

\subsection{A recursion with admissible function}
\label{subsec:admissible_monomial}

In this subsection we define admissible monomials as elements of
$R(X_{\Delta},\mathcal{G})$. This will allow us to define the set $\mathcal{B}(X_{\Delta},\G)$ which will be proven to be an additive basis of $R(X_{\Delta},\mathcal{G})$. The main result is \Cref{lemma:recursion_admissible} that proves a crucial recursion for the proof of \Cref{lemma:H>R}.

\begin{definition}\label{def:f_m}
Given a monomial $m\in R(X_{\Delta},\mathcal{G})$ of the following form:
\[
m=t^{e_1}_{a_1}t^{e_2}_{a_2}\ldots t^{e_k}_{a_k}
\]
where $a_1<\ldots<a_k$ in $\Bl_{\mathcal{G}}(\mathcal{L})$.
We define the function $f_{m} \colon \mathcal{G} \rightarrow \mathbb{N}$ where $a_i=(S_i,X_i)$ as
\[
f_m(G)=\sum_{G\in S_i}e_i.
\]
\end{definition}

Notice that the support of $f_m$ is $S_k\subset \mathcal{G}$ when $m \neq 1$ while $\supp(f_1)= \emptyset$.

\begin{remark}\label{rem:monomial}
Let $m$ be a monomial as in Definition \ref{def:f_m}, let $f_m$ the function associated to it, then

\[
m=t_{a_k}\prod_{G'\in \supp(f_m)}t_{g'}^{f_m(G')-1}=t_{a_k}m'
\]
furthermore, if $t_{g_1}t_{g_2}|m'$ and $G_1\vee G_2=\pi(b)$ in $\mathcal{L}_{\leq \pi(a_k)}$, then 
\[
m=t_{a_k}t_b(m'/t_{g_1}t_{g_2}).
\]
\end{remark} 

%The monomials considered in \Cref{def:admissible} generate the ring $R(X_\Delta,\G)$ as $B$-module.

\begin{lemma}
\label{LCflags}
Any monomial in $R(X_{\Delta},\G)$ can be written as $B$-linear combination of monomials of the form $t_{a_1}t_{a_2}\cdots t_{a_k}$ where $a_1 \leq \cdots \leq a_k$ is a flag of nested sets.
\end{lemma}

\begin{proof}
We prove the statement by induction on the number $n$ of incomparable pairs of nested sets in the generic monomial $t_{a_1}t_{a_2}\cdots t_{a_k}$.
If $n=0$ the monomial is a flag monomial.
If $n>0$, assume without loss of generality that $a_1$ and $a_2$ are incomparable. Then, apply relation (3) in \Cref{def:cohomologyring} to obtain the equation
$$t_{a_1}t_{a_2}t_{a_3}\cdots t_{a_k} = \left(t_{a_1\wedge a_2}\sum_{x\in a_1\vee a_2}t_x\right) t_{a_3}\cdots t_{a_k}.$$
Now each monomial in the right hand side of the equation has strictly less incomparable pairs than the left hand side.
Indeed, for some $a_j$ with $j\neq 1,2$ it is enough to observe that:
\begin{itemize}
\item if $a_1 < a_j$, then $a_1\wedge a_2 < a_j$
\item if $a_1 > a_j$, then $a_j < x$ for any $x\in a_1\vee a_2$,
\end{itemize}
and similarly on $a_2$. If $a_j$ is incomparable with both $a_1$ and $a_2$, then it is also incomparable with $a_1\wedge a_2$ and $x$.
\end{proof}

We extend the definition of \emph{admissible monomial} (of \emph{admissible function}) from \cite{yuzvinsky} and \cite{gaiffiselecta}:

\begin{definition}\label{def:admissible}
A monomial $m\in R(X_\Delta,\mathcal{G})$ is \emph{admissible} if 
\[
m=t^{e_1}_{a_1}t^{e_2}_{a_2}\ldots t^{e_k}_{a_k}\text{ with }a_1<\ldots<a_k
\]
 and for every $G\in\supp(f_m)$
 \[
f_m(G)< \rk(G) -\rk(M_{S_k,G}),
 \]
 where $M_{S_k,G}$ is the element $\bigvee (S_k)_{<G}$ in the sublattice $[\hat{0},G]$.

\end{definition}

By \Cref{thm:cohomology_toric} $H^*(X_{\Delta},\mathbb{Z})\simeq \mathbb{Z}[c_r]_{r\in\mathcal{R}}/L_{\Delta}$, let $\beta_\Delta$ be a 
%universal
Gr\"obner basis 
%\footnote{A universal Gr\"obner basis for an ideal $I\subset R$ is a set of monomials in $I$ which is a Gr\"obner basis according to every possible monomial order.}
for $L_{\Delta}$ and let $\mu_{\Delta}$ be the Gr\"obner escalier of the ideal $L_\Delta$, i.e.\ the set of irreducible monomials with respect to $\beta_\Delta$.
The set $\mu_{\Delta}$ is a monomial basis for the cohomology ring $H^*(X_{\Delta},\mathbb{Z})$. 
We have seen that one can associate with any $a \in \Bl_\G(\LL)$ a fan $\Delta_a:=\Delta\cap\Ann(\Gamma_{\pi(a)})$ which is a subfan of $\Delta$, 
We define the Gr\"obner basis $\beta_{\Delta_a}$ and the monomial basis $\mu_{\Delta_a}$ as the analogous of $\beta_{\Delta}$ and $\mu_{\Delta}$.
%and a map  between toric varieties 
%$j_a \colon X_{\Delta_a}\hookrightarrow X_{\Delta}$. 
%The fact that $\beta_{\Delta}$ is a universal Gr\"obner basis implies that 
%$\beta_{\Delta_a}=\beta_\Delta \cap \mathbb{Z}[c_r,r\in %\Ann(\Gamma_{\pi(a)})]$ is a Gr\"obner basis for the ideal $L_{\Delta}\cap \mathbb{Z}[c_r,r\in \Ann(\Gamma_{\pi(a)})]=L_{\Delta_a}$.

%Let $j^*_{a}$ be the projection map $j^*_a\colon H^*(X_{\Delta};\ZZ) \rightarrow H^*(X_{\Delta_a};\ZZ)$ induced by the above inclusion, by \Cref{prop:restriction_cohomology} we know that $\text{Ker}(j^*_a)$ is generated by the variables $c_r$, $r\notin \Ann(\Gamma_{\pi(a)})$.
%Let $\mu_{\Delta_a}$ be the subset of $\mu_{\Delta}$ of monomials that are uniquely in terms of the variables $c_r$.
%Then $j^*_a(\mu_{\Delta_a})$ is a basis of $H^*(X_{\Delta_a})$.

%We denote by $\Theta(a)$ a minimal set of elements of $H^*(X_{\Delta};\ZZ)$ such that their image through $\rho_a$ is a Groebner basis of 
%$H^*(X_{\Ann(\Gamma_{\pi(a)})};\ZZ)$.

We define the sets $\mathcal{B}(X_\Delta,\G), \mathcal{AM}(X_\Delta,\G)\subset R(X_\Delta,\G)$:
\begin{gather*}
\mathcal{B}(X_\Delta, \G):=\{bm| m \text{ is an admissible monomial}, b\in \mu_{\Delta_{\supp(f_m)}}\};
\\
\mathcal{AM}(X_\Delta, \G):=\{m| m \text{ is an admissible monomial} \};
\end{gather*}
Let $\Gamma_{\mathcal{AM}}(X_\Delta,\mathcal{G})$ (and $\Gamma_{\mathcal{B}}(X_\Delta,\mathcal{G})$) be the generating function of the set $\mathcal{AM}(X_\Delta, \G)$ (respectively of $\mathcal{B}(X_\Delta, \G)$), namely
\begin{gather*}
\Gamma_{\mathcal{AM}}(X_\Delta,\mathcal{G}):=\sum_{i}|\{\text{admissible monomials in }\mathcal{AM}(X_\Delta,\mathcal{G}) \text{ of degree } i\}|y^i, \\
\Gamma_{\mathcal{B}}(X_\Delta,\mathcal{G}):=\sum_{i}|\{bm\in \mathcal{B}(X_\Delta,\mathcal{G}) \text{ of total degree } i\}|y^i.
\end{gather*}
Let $\G:=\{G_1,\ldots,G_r\}$  and $\succ_\G$  a linear order on $\G$ such that $G_i<_{\LL}G_j$ if $j<i$; in the following we denote $G_r$ simply as $G$. 

\begin{example}
\label{runningadmissible}
    Consider the arrangement and building set in \Cref{running}. The maximal chains in $\Bl_{\G}\LL$ are:
    \begin{align*}
        & \{a\}<(\{a,b\},L_i)<\{a,b,P_i\} \text{ for } i=1,2,3, \\
        & \{b\} < (\{a,b\},L_i)<\{a,b,P_i\} \text{ for } i=1,2,3, \\
        & \{c\} < \{c,P_i\} \text{ for } i=1,2,3.\\
        & \{P_i\} < \{a,P_i\} < \{a,b,P_i\} \text{ for } i=1,2,3,\\
        & \{P_i\} < \{b,P_i\} < \{a,b,P_i\} \text{ for } i=1,2,3.\\
        &\{P_i\} < \{c,P_i\} \text{ for } i=1,2,3.
    \end{align*}
    The only admissible monomials are:
    \begin{align*}
        1,t_{\{c\}}, t_{\{P_i\}},t_{\{P_i\}}^2\text{ for } i=1,2,3,
    \end{align*}
    coming respectively from the chains:
    \begin{align*}
        & \emptyset ,\{c\},\{P_i\}\text{ for } i=1,2,3.
    \end{align*}
    The monomial bases for the associated toric varieties are: 
    \begin{align*}
        & \mu_\Delta\text{: a monomial basis for } H^{\ast}(X_{\Delta};\ZZ), \\
        &\mu_{\Delta_{\{c\}}} = \{1,c_3\}, \\
        &\mu_{\Delta_{\{P_i\}}} = \{1\}.
    \end{align*}
\end{example}

\begin{lemma}\label{lemma:recursion_admissible}
The generating function of $\mathcal{AM}$ satisfies the relation
\[
\Gamma_{\mathcal{AM}}(X_\Delta,\mathcal{G})=\Gamma_{\mathcal{AM}}(X_\Delta,\mathcal{G}')+\frac{y(1-y^{d-1})}{1-y}\Gamma_{\mathcal{AM}}(X_{\Delta_G},\mathcal{G}'')
\]
where $d=\rk(G)$, $\G'=\G \setminus \{G\}$, and $\G''=\G_{G}$.
\end{lemma}

\begin{proof}

We divide the set $\mathcal{AM}(X_\Delta,\mathcal{G})$ according to the value of the associated admissible function on $G$:
\[
\mathcal{AM}_i=\{m \in \mathcal{AM}(X_{\Delta},\G)| f_m(G)=i \}
\]
notice that $\mathcal{AM}_i$ is empty for $i\geq d$.
The supports of the elements in $\mathcal{AM}_0$ are nested sets not containing $G$, hence nested sets in $\G'$ and $\mathcal{AM}_0$ is in graded-preserving bijection with $\mathcal{AM}(X_{\Delta},\G')$.
Therefore the contribution of the terms in $\mathcal{AM}_0$ to $\Gamma_{\mathcal{AM}}(X_{\Delta},\G)$ is equal to $\Gamma_{\mathcal{AM}}(X_{\Delta},\G')$.

If $m\in \mathcal{AM}(X_{\Delta},\G)$ and $G\in \supp(f_m)$ then the constraint on $f_m(G)$ is 
\[0<f_m(G)<\rk(G)-\rk(M_{S_k,G})=\rk(G)-\rk(\hat{0})=d\]
the equality $M_{S_k,G}=\hat{0}$ follows from the fact that $G$ is minimal in $\G$ according to the partial order of $\LL$.
The sets $\mathcal{AM}_1,\ldots,\mathcal{AM}_{n-1}$ are in bijection with each other.
More precisely, given $m\in \mathcal{AM}_1$, 
\[
m=t^{e_1}_{a_1}t^{e_2}_{a_2}\ldots t^{e_k}_{a_k}=t_{a_k}\prod_{G_i\in \supp(f_m)}t_{g_i}^{f_m(G_i)-1}\text{ with }a_1<\ldots<a_k \text{ in }\Bl_\G(\LL)
\]
with $g\leq  a_k$;
there exists $\tilde{m}\in \mathcal{AM}_i$, $2\leq i<d$, such that $\tilde{m}=t_g^{i-1}m$.

We want to put the elements in $\mathcal{AM}(X_\Delta,\G'')$ in bijection with the elements of $\mathcal{AM}_1$.
Let $m\in \mathcal{AM}_1$, then
\[
m=t^{e_1}_{a_1}t^{e_2}_{a_2}\ldots t^{e_k}_{a_k}\text{ with }a_1<\ldots<a_k \text{ in }\Bl_{\G}(\LL).
\]
From $f_m(G)=1$ we deduce that $e_k=1$, $g\leq a_k$ and $g\nleq a_{k-1}$ $\Bl_{\G}(\LL)$. We take $b_k=a_k=a_k\vee g$ and $b_i:=(a_i\vee g)_{\leq a_k}$ in $\Bl_{\G}(\LL)$  for $i=1,\ldots, k-1$.
By \Cref{lemma:building_in_contr} we have $\{b_1,\ldots,b_k\}\subset (\Bl_\G(\LL))_{\geq g} \simeq \Bl_{\G''} (\LL_{\geq G})$, we define  
\[
\phi \colon \mathcal{AM}_1\rightarrow \mathcal{AM}(X_{\Delta_G},\G'')
\]
by
\begin{equation}\label{eq:def_phi}
\phi(m)=
\begin{cases}
t_{b_k}^{1}t_{b_{k-1}}^{e_{k-1}}\ldots t_{b_{1}}^{e_{1}} &\text{ if }a_k \not \in a_{k-1}\vee g, \\
t_{b_{k-1}}^{e_{k-1}+1}\ldots t_{b_{1}}^{e_{1}} &\text{ if }a_k \in a_{k-1}\vee g.
\end{cases}
\end{equation}

Next we prove that $\phi$ is well defined i.e.\ $\phi(m)\in \mathcal{AM}(X_{\Delta_G},\G'')$ for every $m\in \mathcal{AM}_1$. Let $S_k=\{G,G_{i_1},\ldots,G_{i_l}\}$, then for every $G_{i_j}$ one of the following holds:
\begin{itemize}
\item $G_{i_j}> G$ in $\LL$ and $G_{i_j}\vee G=G_{i_j}=H_j\in \LL_{\geq G}$. From the definition of nested set, no other $G_{i_{h}}\in S_k$ is such that $G_{i_j}= (G_{i_h}\vee G)_{\leq X_k}$;
\item $G_{i_j}$ and $G$ are incomparable in $\LL$ and $H_{j}=(G\vee G_{i_j})_{\leq X_k}$. No other $G_{i_{h}}$ is such that $H_{j}=(G\vee G_{i_h})_{\leq X_k}$ otherwise $(G\vee G_{i_j})_{\leq X_k}=(G\vee G_{i_h}\vee G_{i_h})_{\leq X_k}$ in contradiction with the fact that $\G$ is building.
\end{itemize}

Assume that $b_k=(T_k,Y_k)$, then there is a $1:1$ correspondence between $T_k$ and $S_k\setminus\{G\}$.

By definition of $\phi$ the equality $f_m(G_{i_j})=f_{\phi(m)}(H_j)$ holds.
In order to see that $\phi(m)\in \mathcal{AM}(X_{\Delta_G},\G'')$ we study the rank of the elements in $\supp(f_{\phi(m)})$.
In the following we denote by $\rk$ the rank in $\LL$ and by ${\rk''}$ the rank in $\LL_{\geq G}$.

\begin{itemize}
\item If $G_{i_j}<G$ and $H_j=G_{i_j}$ we have that $\rk(G_{i_j})=\rk(H_j)$ and $\rk''(H_{j})=\rk(H_j)-d$. Notice that  $G_{i_h}\in S_k $ is such that $G_{i_h}< G_{i_j}$ if and only if $(G_{i_h}\vee G)_{\LL_{X_k}}=H_h< H_{i_j}$. 
Let
\begin{equation}\label{eq:mfG}
M_{S_k,G_{i_j}}=\big(\bigvee \{G_{i_h}\in S_k| G_{i_h}<G_{i_j} \}\big)_{\LL_{X_k}};
\end{equation}
\begin{equation}\label{eq:m'fH}
 M_{\supp(f_{m'}),H_{j}}=\big(\bigvee \{H_{h}\in T_k| H_{h}<H_{j} \}\big)_{\LL_{Y_k}}
\end{equation}
then, because $G< G_{i_j}$, $M_{S_k,G_{i_j}}=  M_{\supp(f_{\phi(m)}),H_{j}}$ and $\rk(M_{\supp(f_{\phi(m)}),H_{j}})=\rk''(M_{\supp(f_{\phi(m)}),H_{j}})+d$.

\item Assume that $G_{i_j}$ and $G$ are incomparable, then $H_{j}=G_{i_{j}}\vee G$ in $\LL_{X_k}$ and, because $G$ and $G_{i_j}$ are the $\G$-factors of $H_{j}$, $\rk(H_{j})=\rk(G)+\rk(G_{i_j})$, $\rk''(H_{i_j})=\rk(G_{i_j})$.
Taking $M_{S_k,G_{i_j}}$ and $M_{\supp(f_{m'}),H_{j}}$ as in eq.\ \eqref{eq:mfG}, \eqref{eq:m'fH} we notice that $G$ and $M_{S_k,G_{i_j}}$ are incomparable and $M_{\supp(f_{m'}),H_{j}}=M_{S_k,G_{i_j}}\vee G$ in $\LL_{\leq X_k}$.  We conclude that $$\rk(M_{\supp(f_{\phi(m)}),H_{j}})=\rk(M_{S_k,G_{i_j}})+\rk(G)$$ and $$\rk''(M_{\supp(f_{\phi(m)}),H_{j}})=\rk(M_{S_k,G_{i_j}}).$$

\end{itemize}

In both situations we obtain
\begin{equation}\label{eq:rank_equality}
f_{\phi(m)}(H_j)=f_{m}(G_{i_j})< \rk(G_{i_j})-\rk{M_{S_k,G_{i_j}}}= 
{\rk''}(H_k)-{\rk''}{M_{\supp(f_{\phi(m)}),H_{j}}}.
\end{equation}

We therefore proved that $\phi \colon \mathcal{AM}_1\rightarrow \mathcal{AM}(X_{\Delta_G},\G'')$ is well defined, furthermore 
 $\phi(m)$ is an admissible monomial whose degree is the degree of $m$ decreased by $1$. At last we show that $\phi$ is a bijection, to this it is crucial to keep in mind that $(\Bl_\G(\LL))_{\geq g} \simeq \Bl_{\G''} (\LL_{\geq G})$ (see \Cref{lemma:building_in_contr}).

We study the injectivity of $\phi$. Assume by contradiction that $\phi(m_1)=\phi(m_2)=m'$ and let $(T,Y)\in (\Bl_\G(\LL))_{\geq g}$ be the support of $f_{m'}$then, because $(\Bl_\G(\LL))_{\geq g} \simeq \Bl_{\G''} (\LL_{\geq G})$, $f_{m_1}$ and $f_{m_2}$ must have the same domain $(S,X)\in \Bl_{\G''} (\LL_{\geq G})$.
We already noticed that there is a bijection between $S\setminus \{G\}$ and $T$, say that this bijection associates $G_{i_j}$ with $T_j$. Based on the definition of $\phi$ we have that $f_{m_1}(G_{i_j})=f_{m'}(T_j)=f_{m_2}(G_{i_j})$ implying that $m_1=m_2$.

To see that $\phi$ is surjective we consider $m'\in \mathcal{AM}(X_{\Delta_{G}},\G'')$ a find $m\in \mathcal{AM}_1$ such that $\phi(m)=m'$. Let $(T,Y)\in (\Bl_\G(\LL))_{\geq g}$ be the support of $m'$, the support of $m$ must be the one nested set $(S,X)\in \Bl_\G{\LL}$ that corresponds to $(T,Y)$ through the isomorphism $(\Bl_\G(\LL))_{\geq g} \simeq \Bl_{\G''} (\LL_{\geq G})$. From the already stated bijection between $T$ and $S\setminus \{G\}$ we obtain the values for $f_m$.
The resulting monomial $m$ is admissible, this can be proved with arguments similar to the ones in the proof of \cref{eq:rank_equality}.

It follows that $\phi$ is a grade-shifting bijection between $\mathcal{AM}_1$ and $\mathcal{AM}(X_{\Delta_G},\G'')$ and in particular the degree is shifted by $-1$. We deduce that the contribution of the sets $\mathcal{AM}_1,\ldots,\mathcal{AM}_{d-1}$ is
\[
\frac{y(1-y^{d-1})}{1-y}\Gamma_{\mathcal{AM}}(X_{\Delta_G},\mathcal{G}'')
\]
and this concludes the proof.
\end{proof}

\begin{lemma}\label{lemma:recursion_gamma}
The generating function of $\mathcal{B}(X_\Delta, \G)$satisfies the relation
\[
\Gamma_{\mathcal{B}}(X_\Delta,\mathcal{G})=\Gamma_{\mathcal{B}}(X_\Delta,\mathcal{G}')+\frac{y(1-y^{d-1})}{1-y}\Gamma_{\mathcal{B}}(X_{\Delta_G},\mathcal{G}'')
\]
where $d=\rk(G)$, $\G'=\G \setminus \{G_r\}$, and $\G''=\G_{G_r}$.
\end{lemma}
\begin{proof}

The proof is similar to the one of \Cref{lemma:recursion_admissible}. We split $\mathcal{B}(X_{\Delta},G)$ in the subsets
$$
\mathcal{B}_i:=\{bm\in \mathcal{B}(X_{\Delta},\G) | f_m(G)=i\}.
$$
We notice that
\begin{itemize}
\item $\mathcal{B}_0$ is in bijection with $\mathcal{B}(X_{\Delta},\G)$;
\item $\mathcal{B}_1,\ldots,\mathcal{B}_{d-1}$. Similarly to what was done for \Cref{lemma:recursion_admissible} we can define a bijection from $\mathcal{B}_1$ to $\mathcal{B}_i$ that associates to
\[
bm=bt_{a_k}\prod_{G_i\in \supp(f_m)}t_{g_i}^{f_m(G_i)-1} 
\text{  with } a_k\in \Bl_{\G}(\LL),
 b\in \mu_{\Ann(\Gamma_{\pi(a_k)})}
\]
the monomial 
\[
bm'=bt_g^{i-1}\in \mathcal{AM}_i.
\]
\end{itemize}

We focus on proving a bijection between $\mathcal{B}_1$ and $\mathcal{B}(X_{\Delta_G},\mathcal{G}'')$.
It is enough to recall that $\Gamma_{\pi(a_k)}$ 
for $a_k=(S,X)\in \Bl_{\G}(\LL)$ is the intersection of the following split direct summands of $X^*(T)$: $	\{\Gamma(G_i)\}_{G_i\in \max S}$.
With this in mind we see that if $\phi(m)=m'$ (see eq.\ \eqref{eq:def_phi} for the definition of $\phi$), then
\[
\Gamma_{\pi''(\supp(f_{m'}))}=\Gamma_{\pi(\supp(f_m))}
\]
where $\pi'' \colon \Bl_{\G''}(\LL'')\rightarrow \LL''$. More in detail we see that if $\supp(f_m)=\{G,G_{i_1},\ldots,G_{i_{\ell}}\}$ and $\supp(f_{m'})=\{H_1,\ldots,H_{\ell}\}$, then $\Gamma(H_i)=\Gamma(G)\cap \Gamma(G_i)$.
It follows that $\mu(\Gamma_{\supp(f_m)}) = \mu(\Gamma_{\supp(f_{m'})})$ from this and the grade shifting bijection $\phi$ from $\mathcal{AM}_1$ and $\mathcal{AM}(X_{\Delta_G},\G'')$ one concludes that
\[
\Gamma_{\mathcal{B}}(X_{\Delta},\mathcal{G})=\Gamma_{\mathcal{B}}(X_{\Delta},\mathcal{G}')+\frac{y(1-y^{d-1})}{1-y}\Gamma_{\mathcal{B}}(X_{\Delta_G},\mathcal{G}''). \qedhere
\]
\end{proof}

\subsection{Additive basis via Gr\"obner basis}
\label{subsec:Grobner_basis}

In what follows we present the proof of \Cref{thm:Grobner_Basis} that describes a Gr\"obner basis for the ideal $I_{\LL}+L_{\Delta}$.
A a corollary of this is the fact that the set $\mathcal{B}(X_{\Delta},\G)$ is an additive basis for $R(X_{\Delta},\G)$.

In order to work with Gr\"obner basis in $\mathbb{Z} [c_r, t_a]$ we fix a monomial order.
We already defined a total order $\succ_\G$ on the building set $\G$ which refines the partial order on $\G \subset \LL$, consider also $\succ_{\mathcal{R}}$ as a fixed total order of $\mathcal{R}$. 
Given two nested sets $a=((S_1 \succ_\G S_2 \succ_\G \dots \succ_\G S_r), x)$ and $b=((T_1 \succ_{\G} T_2 \succ_\G \dots \succ_\G T_k), y) \in \Bl_\G \LL$,
we introduce a partial order on $\Bl_\G \LL$ given by
\begin{align*}
    a \succ b & \text{ if } S_i \succ_\G T_i  \text{ and } S_j = T_j \text{ for all } j<i, \\
    & \text{ or if } r>k \text{ and } S_j=T_j \text{ for all } j\leq k.
\end{align*}
We extend arbitrarily this partial order to a total order $\succ_{\Bl_\G \LL}$ on $\Bl_\G \LL$.
Notice that the total order $\succ_{\Bl_\G \LL}$ refines the partial order $>_{\Bl_\G \LL}$ on the nested sets.
Now we order the variables $t_a$ using the reverse order on the nested sets, i.e.\ $t_a > t_b$ if and only if $a \prec_{\Bl_\G \LL} b$. 
Moreover, we set $t_a > c_r$ and $c_r > c_s$ if $r \succ_{\mathcal{R}} s$.
We introduce the monomial order DegRevLex on $\mathbb{Z} [c_r, t_a]$ using the above order on variables.

In particular the initial ideal of $I_\G+L_{\Delta}$ (see Definition \ref{def:cohomologyring}) contains the monomials:
\begin{enumerate}
%    \item $\prod_{i} c_{r_i}$,
%    \item $t_ac_r$ for all $a \in \Bl_\G(\LL)$ and all $r \not \in \Ann \Gamma_{\pi(a)}$,
    \item $t_bt_G^{s-1}$ for all $a \lessdot b \in \Bl_\G \LL$ such that $\pi(a) < \pi (b) $ and $G\in \G$ is the unique element so that $b \in \{ G \} \vee a$,
    \item $t_at_b$ for all $a,b$ incomparable elements in $(\Bl_\G \LL) \setminus \{\hat{0}\}$.
\end{enumerate}

\begin{remark}
    Recall that there is a theory of Gr\"obner bases for polynomial rings with coefficient in a Noetherian, commutative ring $R$. As we just said, in our case we will be considering $R=\ZZ$. The definition of Gr\"obner basis is analogous to the case of coefficients being in a field, although the theory differs necessarily in some aspects. For details, see \cite[Chapter 4]{AdamsGroebnerBases}.
\end{remark}

\begin{theorem} \label{thm:Grobner_Basis}
Given a toric arrangement $\mathcal{A}$ and a building set $\G\subset \LL(\mathcal{A})$, let $\Delta\subset V$ be a fan and $U$ the set of polynomials described in \Cref{def:cohomologyring}. Then the following set:
$$
\alpha_{\Delta,\G}:=\{t_ab| a\in \Bl_{\G}(\LL) \setminus \{ \hat{0}\}, b\in \beta_{\Delta_a}\}\cup \beta_{\Delta} \cup U
$$
is a Gr\"obner basis for $I_{\G}+L_{\Delta}$ respect to the monomial order DegRevLex.
\end{theorem}
The proof of the previous theorem will be postposed to the end of \Cref{sec:main_proof}.

Let $\Hilb(A)$ be the Hilbert series in the variable $y$ of a graded $\mathbb{Z}$-algebra $A$, where the ``dimension'' of each graded piece is the rank as $\mathbb{Z}$-module.
We assume that the rank of every graded piece of $A$ is finite.
Throughout $J=I_{\G}+L_{\Delta}$ and $\alpha=\alpha_{\Delta,\G}$, the initial ideal $\In(\alpha)=(\operatorname{lt}(p) \mid p \in \alpha)$ of the set $\alpha$ is contained in $\In(J)$, and 
\[\Hilb(\ZZ[c_r,t_a]/\In(\alpha)) \geq \Hilb(\ZZ[c_r,t_a]/\In(J)) = \Hilb (\ZZ[c_r,t_a]/J). \]
The inequality between two series is intended coefficientwise.

\begin{remark} \label{remark:torsion_free}
    Let us assume that $\beta_\Delta$ and $\beta_{\Delta_a}$ for any $a \in \Bl_\G \LL$ are minimal Gr\"obner basis. Then, each polynomial $f \in \alpha_{\Delta,\G}$ has leading coefficient one $\operatorname{lc}(f)=1$. Therefore, the $\ZZ$-module $\ZZ[c_r,t_a]/\In(\alpha_{\Delta,\G})$ is torsion free.

    The same holds for arbitrary Gr\"obner basis: indeed the ideal $\In(\alpha_{\Delta,\G})$ does not depend on the choice of $\beta_\Delta$ and $\beta_{\Delta_a}$'s.
\end{remark}

\begin{lemma} \label{lemma:H>R}
    We have $\Hilb(\ZZ[c_r,t_a]/\In(\alpha_{\Delta,\G}))=\Hilb(H^*(Y(X_\Delta,\G);\ZZ))$.
    In particular,
    \[ \Hilb(H^*(Y(X_\Delta,\G);\ZZ)) \geq \Hilb(R(X_\Delta,\G)). \]
\end{lemma}
\begin{proof}
%Throughout $J=I_{\G}+L_{\Delta}$ and $\alpha=\alpha_{\Delta,\G}$, furthermore we 
Let $M$ be the family of monomials with coefficient 1 in $\ZZ[c_{r}, t_a,r\in\mathcal{R},a\in \Bl_\G(\LL) \setminus \{ \hat{0}\}]$.
%The set $\alpha$ generates the ideal $J$, in order to prove that it is a Gr\"obner basis we study the initial ideal $\In(\alpha)=(\operatorname{lt}(p) \mid p \in \alpha)$ and compare it with $\In(J)$.
%Of course $\In(\alpha)\subseteq \In(J)$ and we also know that
%\[
%\dim (\ZZ[c_r,t_a]/\In(J))=\dim (\ZZ[c_r,t_a]/J). 
%\]
%\todo{in che senso dimensione?}

%In order to prove that $\alpha$ is a Gr\"obner basis 
We show the following two facts: 
\begin{itemize}
\item[a)] $\mathcal{B}(X_{\Delta},\G)=M\setminus (\In(\alpha)\cap M)$;
\item[b)] the dimension of $\Gamma_{\mathcal{B}}(X_{\Delta},\mathcal{G})=\Hilb(H(Y(X_\Delta,\G));\ZZ)$ .
\end{itemize}
These two equalities prove the first statement, then the second follows easily.
%Then we conclude that $$\dim\ZZ[c_{r}, t_a]/\In(\alpha)=\dim\ZZ[c_{r}, t_a]/\In(J)$$ and so $\In(\alpha)=\In(J)$.

\begin{itemize}
\item[Proof of a)]
Let $m=\prod_{a\in A}t^{e(a)}_a\prod_{r\in R}c^{e(r)}_r$ be a monomial in $M\setminus \In(\alpha)\cap M$, 
where $A\subset \Bl_\G(\LL)$ and $R\subset \mathcal{R}$. By \Cref{def:cohomologyring},(3) we have that the elements in $A$ must be a chain in $\Bl_\G(\LL)$, we fix the following notation $A:={a_1<a_2\ldots<a_k}$ according to the order of $\Bl_\G(\LL)$.

We focus on the part in term of the variables $t_a$, let $a_i=(S_i,x_i)$:
as already noticed in \Cref{rem:monomial}
\[
m=t^{e_1}_{a_1}t^{e_2}_{a_2}\ldots t^{e_k}_{a_k}=t_{a_k}\prod_{G_i\in \supp(f_m)}t_{g_i}^{f_m(G_i)-1}\text{ with }a_1<\ldots<a_k \text{ in }\Bl(\LL)
\]
 We want to find inequalities for the exponents $f_m(G_i)$. Fix $G$ such that $f_m(G)>0$. Notice that $G\in S_k$ and if there a certain $H\in S_k$ such that $H>_{\LL}G$, then
\[
 t_{a_k}=t_{a'_k}t_{h}
 \] 
 where $a'_k:=(S_k\setminus{H},x_k)$. With this in mind one can always assume that $G$ is a maximal element in $S_k$. Let $a''_k:=(S''_k,x''_k)$ where $S''_k=S_k\setminus {G}$ and $x''_k=(\vee S''_k)_{\leq x_k}$; assuming that $G$ is maximal in $S_k$ we see that $a_k,a''_k$ and $G$ satisfy the requirements 
 %of $b,a$ and $G$ as in 
 of \Cref{def:cohomologyring} item ii) therefore in order to have $m\in M\setminus I(Lt(\beta_{X}\cap \alpha))$ we must have $f_m(G)<|\Gamma(a_k)/\Gamma(a''_k)|= \rk(G)-\rk(M_{\supp(f_m),G}$
 %\todo{Viola, c'è un typo ma non riesco a capire quale sia. Riesci a sistemare?}
 . We conclude that $\prod_{a\in A}t^{e(a)}_a\in\mathcal{AM}(X,\G)$. 
We notice that by \Cref{def:cohomologyring},(1) $\prod_{r\in R}c^{e(r)}_r\in \mathbb{Z}[c_r|r\in \Ann_{(a_k)}]$, and by the definition of $\alpha$ $\prod_{r\in R}c^{e(r)}_r\in \mu_{\Delta_{a_k}}$.

\item[Proof of b)]
Notice that $Y(X_\Delta,\G)$ is the blowup of $Y(X_\Delta,\G')$ along the subvariety $\overline{\mathcal{K}}'_G \simeq Y(X_{\Delta_G},\G'')$.
From Theorem \ref{thm:cohomology:blowup} we have that
\[
\Hilb(H^*(Y(X_{\Delta},\mathcal{G}))) = \Hilb(H^*(Y(X_{\Delta},\mathcal{G'}))) + y \frac{1-y^{d-1}}{1-y} \Hilb(H^*(Y(X_{\Delta_G},\mathcal{G}'')))
\]
In Lemma \ref{lemma:recursion_gamma} we proved that the same recursion holds for the generating function of $\mathcal{B}(X_{\Delta},\G)$. 
%To prove the thesis it is enough to show that 
The thesis follows by induction with base step
$$
\Gamma_\mathcal{B}(X_{\Delta},\emptyset)=\Hilb(H^*(X_{\Delta})),
$$
%in the particular case of $\G=\emptyset$, 
this follows easily by noticing that 
$\mathcal{B}(X_{\Delta},\emptyset)=\beta_{\Delta}$. \qedhere
%and $R(X_{\Delta},\emptyset)=B=H^*(X_{\Delta},\ZZ)$.
\end{itemize}
\end{proof}

An immediate consequence of \Cref{thm:Grobner_Basis}, \Cref{remark:torsion_free} , and the proof of \Cref{lemma:H>R} is the following:
\begin{corollary}\label{cor:additive_basis}
    The set $\mathcal{B}(X_{\Delta},\G)$ is a $\mathbb{Z}$-basis of the torsion free module $R(X_{\Delta},\mathcal{G})$.
\end{corollary}

\begin{example}
    The Betti numbers of the good toric variety in \Cref{running_fan} are 
    \begin{align*}
        \rk H^0(X_\Delta;\ZZ) &= 1 \\
        \rk H^1(X_\Delta;\ZZ) &= 11 \\
        \rk H^2(X_\Delta;\ZZ) &= 11 \\
        \rk H^3(X_\Delta;\ZZ) &= 1. \\
    \end{align*}
    From the admissible funcions for the running example in \Cref{runningadmissible} and the main \Cref{thm:main}, we get that a monomial basis for the cohomology ring $H^*(Y(X_\Delta,\G);\ZZ)$ is given by the monomial basis for $X_\Delta$, that is $\mu_{\Delta}$, plus the following elements in degree $1$ and $2$ respectively:
    \begin{align*}
        &t_{\{c\}}, t_{\{P_i\}}\text{ for } i=1,2,3 \\
        & c_3t_{\{c\}}, t_{\{P_{i}\}}^{2}\text{ for } i=1,2,3
    \end{align*}
    Thus, the Betti numbers of $Y(X_\Delta,\G)$ are: 
   \begin{align*}
        \rk H^0(Y(X_\Delta,\G);\ZZ) &= 1 \\
        \rk H^1(Y(X_\Delta,\G);\ZZ) &= 15 \\
        \rk H^2(Y(X_\Delta,\G);\ZZ) &= 15 \\
        \rk H^3(Y(X_\Delta,\G);\ZZ) &= 1. \\
    \end{align*}
\end{example}

\section{Proof of \texorpdfstring{\Cref{thm:main}}{Theorem 4.7}}
\label{sec:main_proof}

Now, we use the results of the previous \Cref{sec:additive_basis} to prove our main result.

\begin{proposition}
\label{prop:well_surj}
    Let $\G$ be a building set of cardinality $m$ on a toric arrangement in a $l$-dimensional torus.
    Assume that the map $R(X_\Delta, \tilde{\G}) \to H^*(Y(X_\Delta, \tilde{\G});\ZZ) $ of \Cref{thm:main} is a well defined isomorphism for all $\tilde{\G}$ and all tori $T$ such that $\lvert \tilde{\G} \rvert < m$ or $\dim T < l$.
    Then the map $R(X_\Delta, \G) \to H^*(Y(X_\Delta,\G);\ZZ) $ of \Cref{thm:main} is well defined and surjective.
\end{proposition}

\begin{proof}
\begin{figure}
    \centering
    \begin{tikzcd}
  &   B[(\Bl_\G \LL)^+] \arrow[ldd, "I"', two heads] \arrow[dd, "f", two heads] \arrow[rd] \arrow[rr, "\varphi", hook]               
  &                              & B[(\Bl_{\G'} \LL')^+] \otimes_\ZZ B''[(\Bl_{\G''} \LL'')^+][t_G] \arrow[d, "{I'\otimes 1, 1 \otimes I''}", two heads] \\
  &                                                                                & A \arrow[r, phantom, sloped, "\subset"] \arrow[ld, "J", two heads] & R' \otimes_\ZZ R''[t_G]                                \\
R & H^*(Y(X,\G);\ZZ)                                                                           &                              &                                      
\end{tikzcd}
    \caption{Algebras and morphisms involved in the proof of {\Cref{thm:main}}.}
    \label{fig:diag_dim}
\end{figure}
We will prove the statement by induction on $\lvert \G \rvert$ and $\dim T$.
For the sake of notation we denote $B[t_a \mid a \in (\Bl_\G \LL) \setminus \{\hat{0}\}]$ by $B[(\Bl_\G \LL)^+]$ and $H^*(X_{\Delta_{G}},\mathbb{Z})$ as $B''$.
    Let $G$ be the last element of $\G$ and $g=\{G\}$ be the corresponding atom in $\Bl_\G \LL$.
    Define $\LL' \subset \LL$ as the subset of elements $F$ such that $G$ is not a $\G$-factor of $F$ and set $\G'= \G \setminus \{G\}$.
    Notice that $\G'$ is a building set for the poset $\LL'$.

    Set $\LL'' = \LL_{\geq G}$ and $\G''= \G_G$ as in \Cref{lemma:building_in_contr}.
    We have $\Bl_\G(\LL)= \Bl_{\G'}(\LL') \sqcup \Bl_{\G''}(\LL'')$ because $\Bl_{\G'}(\LL')= \Bl_{\G}(\LL)_{\not \geq g}$ and $\Bl_{\G''}(\LL'') \simeq \Bl_{\G}(\LL)_{\geq g}$. Throughout $I=I_{\G}$, $I'=I_{\G'}$ and $I''=I_{\G''}$ as in \Cref{def:cohomologyring}.
        %\item $B[t_a \mid a \in (\Bl_\G \LL) \setminus \{\hat{0}\}] \simeq B[t_a \mid a \in (\Bl_{\G'} \LL') \setminus \{\hat{0}\}] \otimes_B B[t_a \mid a \in (\Bl_{\G''} \LL'') \setminus \{\hat{0}\}][t_{G_m}]$
    %\end{enumerate}

Consider the ring homomorphism
\[ \varphi \colon B[(\Bl_\G \LL)^+] \to 
%B[(\Bl_\G \LL)^+_{\not \geq G}] \otimes_\ZZ B''[(\Bl_\G \LL)_{> G}][t_G] = 
B[(\Bl_{\G'} \LL')^+] \otimes_\ZZ B''[(\Bl_{\G''} \LL'')^+][t_g] \]
defined by $\varphi(c_r) = c_r \otimes 1$, $\varphi(t_a)= t_a \otimes 1$ if $a \not \geq g$, and $\varphi(t_a) = 1 \otimes t_a t_g$ if $a \geq g$ (with particular case $\varphi(t_g)= 1 \otimes t_g$).

Notice that $\varphi$ is injective because $B''$ is a free $\ZZ$-module and homogeneous because the isomorphism $\Bl_{\G''}(\LL'') \simeq \Bl_{\G}(\LL)_{\geq g}$ shifts the rank by one.
Consider the canonical projection $p \colon B[(\Bl_{\G'} \LL')^+] \otimes_\ZZ B''[(\Bl_{\G''} \LL'')^+][t_G] \to R' \otimes R''[t_G]$ with kernel generated by $I' \otimes 1$ and $1 \otimes I''$.

Consider $Y(X_\Delta, \G)$ as the blow up of $Y(X_\Delta, \G')$ along the layer $\overline{\mathcal{K}}'_G \simeq Y(X_{\Delta_G}, \G'')$. Recall from \Cref{thm:blowup} the subalgebra $A \subset R' \otimes R''[t_G]$ and the projection $q \colon A \to H^*(Y(X,\G); \ZZ)$.
The range of $p\varphi$ is contained in the subalgebra $A$ (c.f.\ \Cref{thm:blowup}) and hence define $f \colon B[(\Bl_\G \LL)^+] \to H^* (Y(X,\G); \ZZ)$ to be the composition $f=qp\varphi$.

\begin{lemma}
    The morphism $f$ is surjective.
\end{lemma}
\begin{proof}
    The generators of type $c_r \otimes 1$ and $t_a\otimes 1$ for $a \not \geq g$ are clearly in the range of the map.
    It enough to find a inverse image of $1\otimes mt_g$ with $m$ a monomial in the variables $t_a$ for $a\geq g$ and $c_r$ for $r \in \Ann(\Gamma_{G})$. By definition of the ideal $J$ and by the identity $\varphi(c_r)=c_r$, we can reduce to the case where $m$ is a monomial only in the variables $t_a$ for $a\geq g$.
    
    Consider a generator of type $1 \otimes t_at_bt_g$ for $a\geq b > g$ and let $c$ be the unique element such that $c \not \geq g$ and $b \in c \vee g$.
    \[ 1 \otimes t_at_bt_g = 1 \otimes t_a \sum_{b' \in c \vee g} t_{b'} t_g = t_c \otimes t_a t_g = f(t_ct_a).\]

\begin{comment}
    Consider $1\otimes t_at_bt_g$ with $a,b > g$ non necessarily comparable. For any $z > g$, let $c_z$ be the unique element such that $z\in c_z\vee g$  with $c_z \ngeq g$. Then:
    \begin{align*}
        1\otimes t_at_bt_g &= 1\otimes t_{a\wedge b} \sum_{x\in a \vee b} t_x \ t_g\\
        &= 1\otimes \sum_{x\in a\vee b}t_x\sum_{y\in c_{a\wedge b}\vee g}t_y \ t_g\\ 
        &= t_{c_{a\wedge b}}\otimes \sum_{x\in a\vee b} t_x \ t_g = f\left(t_{c_{a\wedge b}}\sum_{x\in a\vee b}t_x\right)
    \end{align*}
    This specializes to the case $a\geq b > g$. In general, one can see by induction that:
    \begin{align*}
        1\otimes\prod_{j=1}^{n+1}t_{a_j} = f\left(t_{c_{a_1\wedge a_2}}\sum_{x_1\in a_1\vee a_2} t_{c_{x_1 \wedge a_3}}\cdots \sum_{x_{n}\in x_{n-1}\vee a_n}t_{c_{x_{n-1}\wedge a_{n+1}}} \ \sum_{z \in a_1\vee\cdots\vee a_{n+1}}t_z \right)
    \end{align*}
\end{comment}
Inductively, given a flag $a_1 \leq \ldots  \leq a_k$ with $a_j > g$ for all $j$, let $c_{a_j}$ be the unique element with $c_{a_j}\ngeq g$ and $a_j\in c_{a_j}\vee g$; we have the identity
$$1\otimes t_{a_1}\cdots t_{a_k}t_g = f(t_{c_{a_1}}\cdots t_{c_{a_{k-1}}}t_{a_k}).$$
For a general element $1\otimes t_{a_1}\cdots t_{a_k}t_g$ with $a_1,\ldots,a_k$ not a flag, by \Cref{LCflags} it can be written as linear combination of products on flags, and therefore it belongs to the range of $f$.
\end{proof}

%It is enough to prove that $I=\ker f$.

\begin{lemma} \label{lemma:conti_iota}
  %  If $G \not \leq \pi(a)$, then 
  %  \[t_a \tau_G = \sum_{\substack{c \gtrdot a \\ \pi(c) \geq G}} t_c.\]
    If $F \neq G$, $c> g$, and $\pi(c)\geq F$ then
    \[ \varphi(\tau_F^k t_c) \equiv 1 \otimes  \tau^{\prime \prime k}_{F''}t_c t_g \mod J+ I'\otimes 1 + 1 \otimes I''\]
    where $F'' \in F \vee G$ is the unique element such that $\pi(c) \geq F''$.
\end{lemma}

\begin{proof}
Notice that 
\[ (\G_{\geq F})''= \G''_{\geq F \vee G} := \bigsqcup_{K \in F \vee G} \G''_{\geq K}.\]
We have the equalities
    \begin{align*}
        \varphi(\tau_F^k t_c) &
        %= \left(-\sum_{H\in\G_{\geq F}} t_{\{H\}}\right)^k t_c \\ &
        = (-\sum_{H\in\G_{\geq F}}t_{\{H\}})^k\otimes t_c t_g \\
        & \equiv 1\otimes (-\sum_{H\in\G_{\geq F}}\sum_{H^{\prime\prime}\in G\vee H }t_{\{H''\}})^k t_ct_g \mod J\\ & = 1\otimes (-\sum_{H''\in\G''_{\geq F''}}t_{\{H''\}}- \sum_{H''\in(\G_{\geq F})^{''}_{\ngeq F''}}t_{\{H''\}})^k t_ct_g \\
        &= 1\otimes \sum_{\ell = 0}^{k}\binom{k}{\ell}(-\sum_{H''\in\G''_{\geq F''}}t_{\{H''\}})^{k-\ell}(-\sum_{H''\in (\G_{\geq F})^{''}_{\ngeq F''}}t_{\{H''\}})^\ell t_ct_g \\ & \equiv 1\otimes ( - \sum_{H''\in\mathcal{G''}_{\geq F''}} t_{\{H''\}})^k t_c t_g \mod 1\otimes I'' \\ & =  1\otimes \tau^{\prime\prime k}_{F''} t_c t_g.
    \end{align*}
    The first equality holds by the definition of $\varphi$, the second congruence by \Cref{lemma:pullback}, the fourth equality is the binomial expansion.
    The fifth equivalence is due to the fact that for $H''\in(\G_{\geq F})^{\prime\prime}$ with $H''\ngeq F''$, one has $\{H''\}\vee c = \emptyset$, and hence: 
    \[t_{\{H''\}}t_c \equiv t_{\{H''\}\wedge c}\sum_{d\in\{H''\}\vee c} t_d = 0 \mod I''.\]
    This completes the proof.
\end{proof}

\begin{lemma}
    The ideal $I$ is contained in $\ker f$.
\end{lemma}
\begin{proof}
    We will show that $\varphi(I)\in J + I'\otimes 1 + 1 \otimes I''$.
    \begin{enumerate}
        \item Let $a \in \Bl_\G \LL \setminus \hat{0}$, $r \not \in \Ann(\Gamma_{\pi(a)})$ and consider $\varphi(c_rt_a)$:
        \begin{enumerate}
            \item if $a \not \geq g$ then 
            \begin{align*}
                \varphi(c_rt_a) &=c_rt_a \otimes 1 & & \\
                & \equiv 0 & & \mod I' \otimes 1
            \end{align*}
            \item if $a \geq g$ and $r \not \in \Ann(\Gamma_{G})$ then 
            \begin{align*}
                \varphi(c_rt_a) &=c_r\otimes t_at_g & & \\
                & \equiv 0 & & \mod J
            \end{align*}
            \item if $a \geq g$ and $r \in \Ann(\Gamma_{G})$ then 
            \begin{align*}
                \varphi(c_rt_a) &=c_r\otimes t_at_g & & \\
                & \equiv 1 \otimes c_rt_at_g & & \mod J \\
                & \equiv 0 & & \mod 1 \otimes I''
            \end{align*}
        \end{enumerate}
        \item Let $a \lessdot b$, $F$ be the unique element such that $b \in \{F\} \vee a$ and consider the element $Q_a^b(\G)$:
        \begin{enumerate}
            \item If $a \geq g$ then 
            \begin{align*}
                \varphi(Q_a^b(\G)) &= -\sum_{\substack{c \gtrdot a \\ \pi(c) \geq \pi(b)}} \varphi (t_c \tau_F^{d-1} )  + (1 \otimes t_a t_g) \prod_{j=1}^{d}\left(- \sum_{r\in\mathcal{R}} \min\left(0, \langle\chi_j, r\rangle\right) c_r \otimes 1 \right) \\
                &\equiv -\sum_{\substack{c \gtrdot a \\ \pi(c) \geq \pi(b)}} \varphi (t_c \tau_F^{d-1} )  + 1 \otimes t_a t_g \prod_{j=1}^{d}\left(- \sum_{r\in\mathcal{R}''} \min\left(0, \langle\chi_j, r\rangle\right) c_r \right)
            \end{align*}
            the first summand is 
            \begin{align*}
                \sum_{\substack{c \gtrdot a \\ \pi(c) \geq \pi(b)}} \varphi (t_c \tau_F^{d-1} ) & \equiv 1 \otimes \Bigl(\sum_{\substack{c \gtrdot a \\ \pi(c) \geq \pi(b)}}  t_c t_g  \Bigr)  \tau_{F''}^{\prime \prime d-1}
            \end{align*}
            by \Cref{lemma:conti_iota} where $F'' \in F \vee G$ and $F'' \leq \pi (b)$.
            We obtained the identity
            \[ \varphi(Q_a^b(\G)) \equiv 1 \otimes Q_a^b(\G'') t_g \in 1 \otimes I''.\]
            
            \item If $F \neq G$ and $a \not \geq g$, then $b \not \geq g$ and
            \begin{align*}
                \varphi(Q_a^b(\G)) &=Q_a^b(\G') \otimes 1 \in I' \otimes 1.
            \end{align*}
            
            \item If $F=G$ then $a \not \geq g$ and 
            \begin{align*}
                \varphi(Q_a^b(\G)) =& \Bigl( - 1 \otimes t_bt_g -\sum_{\substack{c \gtrdot a \\ \pi(c) > \pi(b)}} t_c \otimes 1  \Bigr) \Bigl(-1 \otimes t_g - \sum_{H \in \G_{>G}} t_h \otimes 1 \Bigr)^{d-1} + \\
                &+t_a \prod_{j=1}^{d}\left(- \sum_{r\in\mathcal{R}} \min\left(0, \langle\chi_j, r\rangle\right) c_r \right) \otimes 1
            \end{align*}
            Notice that $\G_{>G}$ is a subset of $\G''$, hence for each $H \in \G_{>G}$ makes sense the notation $1 \otimes t_h$, where $h = \{H\}$ is the $\G''$-nested set associated to $H$.
            Moreover, for such $h$ we have $\iota^*(t_h)=t_h$.
            
            The term $(-1 \otimes t_g - \sum_{H \in \G_{>G}} t_h \otimes 1)^{d-1}$ is equal to 
            \begin{align*}
                &(-1)^{d-1}\sum_{i=0}^{d-1} \binom{d-1}{i} (\sum_{H \in \G_{>G}} t_h)^{d-1-i} \otimes t_g^i \\
                &= (-1)^{d-1}\sum_{i=1}^{d-1} \binom{d-1}{i} 1 \otimes (\sum_{H \in \G_{>G}} t_h)^{d-1-i}  t_g^i +  (-\sum_{H \in \G_{>G}} t_h)^{d-1} \otimes 1.
            \end{align*}
            Moreover, $\sum_{\substack{c \gtrdot a \\ \pi(c) > \pi(b)}} t_c \otimes t_g = 1 \otimes t_bt_g \sum_{H \in \G_{>G}} t_h$.
            Hence, by using \Cref{lemma:pushforward} we have
            \begin{align*}
                \varphi(Q_a^b(\G)) =& (-1)^d\sum_{i=0}^{d-1} \binom{d-1}{i} 1 \otimes (\sum_{H \in \G_{>G}} t_h)^{d-1-i}  t_bt_g^{i+1}+ \\
                &+ (-1)^d\sum_{i=1}^{d-1} \binom{d-1}{i} 1 \otimes (\sum_{H \in \G_{>G}} t_h)^{d-i} t_b t_g^i + \iota_*(t_b) \otimes 1 \\
                =& \sum_{i=1}^{d} 1 \otimes t_b \binom{d}{i} \Bigl(- \sum_{H \in \G_{>G}} t_h \Bigr)^{d-i} (-t_g)^i + \iota_*(t_b) \otimes 1 \in J.
            \end{align*} 
           where the last term is the polynomial appearing in eq.\ \eqref{eq:Chern_poly_Y}.
        \end{enumerate}
        \item Let $a,b \in \Bl_\G \LL \setminus \hat{0}$ be two incomparable elements and consider $f(t_at_b - t_{a \wedge b} \sum_{c \in a \vee b} t_c)$:
        \begin{enumerate}
            \item if $a,b \not \geq g$ then 
            \begin{align*}
                \varphi(t_at_b - t_{a \wedge b} \sum_{c \in a \vee b} t_c) &= (t_at_b - t_{a \wedge b} \sum_{c \in a \vee b} t_c) \otimes 1 \in I' \otimes 1
            \end{align*}
            \item if $a,b \geq g$ then
            \begin{align*}
                \varphi(t_at_b - t_{a \wedge b} \sum_{c \in a \vee b} t_c) &= 1 \otimes (t_at_b - t_{a \wedge b} \sum_{c \in a \vee b} t_c) t_g^2  \in 1 \otimes I''.
            \end{align*}
            \item if $a \not \geq g$ and $b = g$ then 
            \[\varphi(t_at_g - \sum_{c \in a \vee g} t_c) = t_a \otimes t_g - 1 \otimes \iota^*(t_a) t_g \in J.\]
            \item If $a \not \geq g$ and $b > g$ then 
            \begin{align*}
                \varphi(t_at_b - t_{a \wedge b} \sum_{c \in a \vee b} t_c) &= t_a \otimes t_bt_g - t_{a \wedge b}\otimes \sum_{c \in a \vee b} t_c t_g & & \\
                & \equiv 1 \otimes (\sum_{a' \in a \vee g} t_{a'} t_bt_g - \sum_{c \in a \vee b}  \sum_{e \in (a \wedge b) \vee g} t_e t_c t_g ) && \mod J \\
                & \equiv 1 \otimes (\sum_{a' \in a \vee g} t_{a'} t_bt_g - \sum_{c \in a \vee b}  \sum_{\substack{e \in (a \wedge b) \vee g \\ e \leq c}} t_e t_c t_g ) && \mod 1 \otimes I'' \\
                & = 1 \otimes (\sum_{a' \in a \vee g} t_{a'} t_bt_g - \sum_{c \in a \vee b}  \sum_{\substack{e \in (a\vee g) \wedge b \\ e \leq c}} t_e t_c t_g ) && \\
                & = 1 \otimes \sum_{a' \in a \vee g} (t_{a'} t_b - t_{a' \wedge b} \sum_{c \in a' \vee b}   t_c) t_g && \\
                & \equiv 0 && \mod 1 \otimes I''.
            \end{align*}
        \end{enumerate}
    \end{enumerate}
    We have dealt with all possible cases.
\end{proof}
This completes the proof of \Cref{prop:well_surj}
\end{proof}

\begin{proof}[Proof of {\Cref{thm:main}} and of \Cref{thm:Grobner_Basis}]
By \Cref{prop:well_surj} there exists a surjective map
$ R \twoheadrightarrow H^*(Y(X,\G);\ZZ)$, and so $\Hilb(R) \geq \Hilb (H^*(Y(X,\G)))$.
Together with \Cref{lemma:H>R}, we obtain 
\[\Hilb(R) = \Hilb (H^*(Y(X_\Delta,\G)))= \Hilb \left( \faktor{\ZZ[c_r,t_a]}{\In(\alpha_{\Delta,\G})} \right)\]
%\todo{riferirsi al lemma corretto} and by eq.\ \eqref{eq:additive_cohomology} we deduce that 
\Cref{thm:Grobner_Basis} follows from the equality of the first and last term and together with \Cref{remark:torsion_free}.

In particular, $R$ is a free $\ZZ$-module. 
Finally, the surjection $ R \twoheadrightarrow H^*(Y(X,\G);\ZZ)$ is indeed an isomorphism.
\end{proof}

\section{A remark on the strata of wonderful model}
\label{sec:strata}

Let $\G^+=\G \sqcup \mathcal{R}$ be the set that parameterize all irreducible components of $Y(X_\Delta,\G) \setminus M_\AAA$.
Recall that, for $a \in \Bl_\G \LL$, $D_a$ is the layer parameterized by $a$.
Moreover, we consider $D_r \subset Y(X_\Delta,\G)$, i.e.\ the strict transform of toric divisor $H_r$ associated to the ray $r\in \mathcal{R}$.
For $a\in \Bl_\G (\LL)$ and $S \subset \mathcal{R}$, the intersection $D_{a,S}:=D_a \cap \bigcap_{r \in S} D_r$ is non-empty if and only if the set $S$ span a cone and $S \subset \Ann \Gamma_{\pi(a)}$.
If $D_{a,S}$ is non-empty then it is connected by the equal-sign property of $X$.

For any lattice $\Lambda$ we set $\Lambda_\CC= \Lambda \otimes_\ZZ \CC$. Given two layers $G>F \in \LL$ we consider the normal bundle $N_F^G := (N_{\overline{G}} \overline{F})_p$ at a point $p \in G$. Since $N_F^G$ can be identified canonically (via translation to the identity of $G$) with 
\[N_F^G = \faktor{\Ann(\Gamma_F)_\CC}{\Ann(\Gamma_G)_\CC}, \]
we omit the dependency of the point $p$.

%For each $G\in \G$, consider the map $h_G \colon M_\AAA \to N^G_{\hat{0}} \setminus \{0\}$ obtained as the composition
%\[ M_\AAA \subset T \to T/G \subset N^G_{\hat{0}}\]
%where the last inclusion is the restriction of translation $\CC^{\codim G} \to \CC^{\codim G}$ by the vector $(-1,-1, \dots, -1)$.

\begin{theorem}[{{\cite[Definition 1.1, Theorem 1.3]{Li09}}}]
    If $\G \neq \emptyset$, then the wonderful model $Y(X_\Delta,\G)$ is the closure of the image of 
    \[ M_\AAA \hookrightarrow \prod_{G \in \G} \Bl_{\overline{G}} X_\Delta.\]
    %\[ M_\AAA \hookrightarrow X_\Delta \times \prod_{G \in \G} \mathbb{P}(N_{\hat{0}}^G)\]
    %where on the first component is the canonical inclusion $M_\AAA \subset T \subset X_\Delta$ and on the other components the map $h_G$.
\end{theorem}
%\begin{proof}
%    The result follows easily from the definition of toric wonderful model \cite{DeConciniGaiffi2018}, the characterization of Li \cite[Definition 1.1, Theorem 1.3]{Li09}, and the observation that the normal bundle of $\overline{G}$ in $X_\Delta$ is trivial.
%    The product in \cite[Definition 1.1]{Li09} specialized to our case has redundant copies of $X_\Delta$ and $\overline{G}$ for $G \in \G$. Getting rid of these factors we obtain the above description.
%\end{proof}

\begin{definition}
    Let $S \subset \mathcal{R}$ be a cone in $V$ and consider the projection $p_S \colon V \to V/\operatorname{span}(S)$. The set 
    \[ \{p_S(C) \mid C \in \Delta \textnormal{ and } C \supset S\}\]
    form a smooth complete fan that we denote by $\Delta_S$.
    Let $\G(S)$ be the set
    \[ \G(S) := \{ G \in \G \mid \Ann(\Gamma_G) \supset S\}. \]
\end{definition}
    Notice that the map $p_S$ corresponds to the inclusion of the orbit closure $\overline{\mathcal{O}}_S \simeq X_{\Delta_S}$ of $S$ in the toric variety $X_{\Delta}$.
    Moreover, by the equal sign property, the intersection of the hypertorus closure $\overline{\mathcal{K}_{\Gamma, \phi}}$ with the orbit $\mathcal{O}_S$ is a toric arrangement with building set $\G(S)$.

\begin{lemma}
\label{lemma:strata}
    For any $\G^+$-nested set $(a,S)$ the intersection $D_{a,S}$ is
    %\[ D_a \simeq Y(X_{\Delta \cap \Ann \Gamma_{\pi(a)}}; \G_{\pi(a)}) \times \prod_{i=1}^l Y^{\PP}(\PP(N_{K_i}^{G_i}), (\G_{\leq G_i})_{K_i})\]
    \[ D_{a,S} \simeq \prod_{i=1}^l Y^{\PP}(\PP(N_{K_i}^{G_i}), (\G_{\leq G_i})_{K_i}) \times Y(X_{\Delta_S \cap \Ann \Gamma_{\pi(a)}}; \G(S)_{\pi(a)})\]
    where $D_a$ is a connected component of $\bigcap_{i=1}^l D_{G_i}$ and let $K_i$ be the unique element of $\bigvee_{G_j < G_i} G_j$ such that $K_i < \pi(a)$
    %\todo{o equivalentemente $K_i< G_i$}.
\end{lemma}
%\begin{proof}
%    Firstly consider the case $S=\emptyset$.
%    Now, we consider an arbitrary stratum $D_{a,S}= D_a \cap D_S$.
%    As observed at the beginning of this section, the stratum $D_{a,S}$ is a stratum in the subvariety $D_S$. Since $D_S\simeq Y(X_{\Delta_S}, \G(\Gamma_S))$, we apply the first part of the proof to the toric arrangement induced on the torus $\mathcal{O}_S$. The result follows easily.
%\end{proof}
\begin{proof}
    Firstly, we show that 
    \[ D_H \simeq Y^{\PP}(\PP(N_{\hat{0}}^H), \G_{\leq H}) \times Y(X_{\Delta \cap \Ann \Gamma_H}; \G_H) .\]
    Recall that the closure of $H$ in $X_\Delta$ is $\overline{H} \simeq X_{\Delta \cap \Ann \Gamma_H }$ (see \cite[Section 3]{DeConciniGaiffi2018}).
    Let $Y= Y(X_\Delta;\G)$ as subset of $\prod_{G \in \G} \Bl_{\overline{G}} X_\Delta$ and $p_G$ be the projection on the factor $\Bl_{\overline{G}} X_\Delta$.
    Notice that
    \begin{equation}
    \label{eq:proj_D_H}
        p_G(D_H)= \begin{cases}
        \Bl_{\overline{G\cap H}} \overline{H} & \text{if } G \not \leq H, \\
        \overline{H} \times \PP(N_{\hat{0}}^G) & \text{if } G \leq H
    \end{cases}
    \end{equation}
    because the normal bundle to $\G$ is trivial and the exceptional divisor is $\overline{G}\times \PP(N_{\hat{0}}^G)$.
    Consider $p_1$ and $p_2$ the projection from $\prod_{G \in \G} \Bl_{\overline{G}} X_\Delta$ to $\prod_{G \in \G_{ \leq H}} \Bl_{\overline{G}} X_\Delta$ and to $\prod_{G \in \G_{\not \leq H}} \Bl_{\overline{G}} X_\Delta$ respectively.
    By eq.\ \eqref{eq:proj_D_H} we have $p_1(D_H) \subseteq \prod_{G \in \G_{\leq H}} \overline{H} \times \PP(N_{\hat{0}}^G)$ and $p_2(D_H) \subseteq \prod_{G \in \G_{\not \leq H}} \Bl_{\overline{G \cap H}} X_{\Delta \cap \Ann \Gamma_H}$.
    
    Motivated by the wonderful model of subspace arrangements \cite{DeConciniProcesi95}, we regard $N_{\hat{0}}^G$ as the quotient $N_{\hat{0}}^H/ N_{G}^H$.
    The subspace $p_1(D_H)$ is contained in $\prod_{G \in \G_{\leq H}} \overline{H} \times \PP(N_{\hat{0}}^G)$ that is mapped to $\prod_{G \in \G_{\leq H}} \PP(N_{\hat{0}}^G)$ and the image of $p_1(D_H)$ is contained in the hyperplane wonderful model $Y^{\PP}(\PP(N_{\hat{0}}^H)), \G_{\leq H}) \subset \prod_{G \in \G_{\leq H}} \PP(N_{\hat{0}}^G)$.
    
    On the other hand, the blow-down maps induce a map $\Bl_{\overline{G \cap H}} X_{\Delta \cap \Ann \Gamma_H} \hookrightarrow \prod_{L \in G\vee H} \Bl_{\overline{L}} X_{\Delta \cap \Ann \Gamma_H}$ and by getting rid of multiple factors we obtain an embedding
    \[ p_2(D_H) \hookrightarrow \prod_{L \in \G_H} \Bl_{\overline{L}} X_{\Delta \cap \Ann \Gamma_H}.\]
    Since $p_2(D_H)$ is the closure of the inverse image of $H \setminus \bigcup_{G \in \G_{\not \leq H}} G= H \setminus \bigcup_{L \in \G_H} L$ along the projection $\prod_{G \in \G_{\not \leq H}} \Bl_{\overline{G}} X_\Delta \to X_\Delta$, we have $p_2(D_H) \hookrightarrow Y(X_{\Delta \cap \Ann \Gamma_H}; \G_H)$.

    We have shown that $p_1$ and $p_2$ induce a map $D_H \to Y^{\PP}(\PP(N_{\hat{0}}^H)), \G_{\leq H}) \times Y(X_{\Delta \cap \Ann \Gamma_H}; \G_H)$. This map is injective and the two varieties are smooth, connected, and of the same dimension. In particular they are isomorphic.

    Let $a$ be a nested set, we claim that
    \[ D_a \simeq \prod_{i=1}^l Y^{\PP}(\PP(N_{K_i}^{G_i}), (\G_{\leq G_i})_{K_i}) \times Y(X_{\Delta \cap \Ann \Gamma_{\pi(a)}}; \G_{\pi(a)})\]
    Let $H\in \G$ be a minimal element in the nested set $a$, and $a'$ be the corresponding nested set in $\G_H$.
    We have $D_a \subset D_H \simeq Y^{\PP}(\PP(N_{\hat{0}}^H)), \G_{\leq H}) \times Y(X_{\Delta \cap \Ann \Gamma_H}; \G_H)$, and denoted by $\pi_2$ the projection on the second factor $\pi_2(D_a) \subseteq D_{a'}$.
    The claimed isomorphism follows by induction, noticing that $(\G_H)_{\pi(a')}= \G_{\pi(a)}$ and $\Ann \Gamma_{\pi(a)} = \Ann \Gamma_H \cap \Ann \Gamma_{\pi(a')}$.

    Finally, the general case follows from 
    \[D_S \simeq Y(X_{\Delta_S};\G(S))\]
    that is stated in \cite[Theorem 5.1]{DeConciniGaiffi2019}.
 \end{proof}

\Cref{lemma:strata} predicts that the toric wonderful models form a module over the operad of hyperplane wonderful models (cf.\ \cite{Coron23} for exact definition of ``operads'' associated to matroids).

We finish the article with the example of a divisorial arrangements of subvarieties in which the poset of layers is a local lattice but not a geometric poset.
Our techniques can be applied to more general cases, such as the one in the following example.

\begin{example}
\label{nonposetoflayers}
    Consider two smooth quadrics $Q_1, Q_2$ in $\mathbb{P}^3$ which intersection is a line $\ell_0$ and a cubic $C$.
    Let $H$ be an hyperplane containing $\ell_0$ and not tangent neither to $Q_1$ nor $Q_2$.
    Let $\ell_1$ and $\ell_2$ the lines in $H \cap Q_1$ and $H \cap Q_2$ different from $\ell_0$.
    Define the three distinct points $q_1=\ell_0 \cap \ell_1$,  $q_2=\ell_0 \cap \ell_2$ and  $p=\ell_1 \cap \ell_2$ and notice that $H \cap C = p \cup q_1 \cup q_2$.
    %For a generic hyperplane $H$ we also have $q_1 \neq q_2$.
    Let $X=\Bl_{q_1 \cup q_2}(\mathbb{P}^3)$ and $\tilde{H}$, $\tilde{Q}_i$, $\tilde{\ell}_{j}$ be the strict transform of the corresponding subvarieties.
    The arrangement $\{\tilde{H}, \tilde{Q}_1, \tilde{Q}_2\}$ in $X$ is a collection of smooth connected divisors which union is a normal crossing divisor.
    However, the intersection $\tilde{H} \cap  \tilde{Q}_1 \cap \tilde{Q}_2 = \tilde{\ell}_0 \sqcup p$ is not equidimensional and the poset of flats is not a geometric poset.

    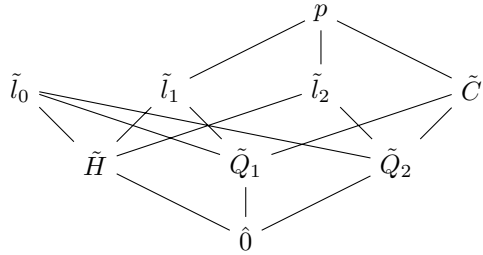
\begin{figure}
        \centering
        \begin{tikzpicture}
\node (min) at (0,-1) {$\hat0$};
\node (H) at (-2,0) {$\tilde{H}$};
\node (Q1) at (0,0) {$\tilde{Q}_1$};
\node (Q2) at (2,0) {$\tilde{Q}_2$};

\node (l0) at (-3,1) {$\tilde{l}_0$};
\node (l1) at (-1,1) {$\tilde{l}_1$};
\node (l2) at (1,1) {$\tilde{l}_2$};
\node (C) at (3,1) {$\tilde{C}$};

\node (p) at (1,2) {$p$};

\draw (min) -- (Q2) -- (l2) -- (H) -- (l1) -- (Q1) -- (l0) -- (H) -- (min) -- (Q1) -- (C) -- (Q2) -- (l0);
\draw (l1) -- (p) -- (l2);
\draw (C) -- (p);
\end{tikzpicture}
        \caption{The poset of flats of an arrangement of subvarieties.}
        \label{fig:non_geometric_poset}
    \end{figure}
\end{example}

\bibliographystyle{amsalpha}
%\nocite*
\bibliography{bibtoric}
\end{document}